\documentclass[reqno]{amsart}
\usepackage{amsmath}
\usepackage{amsthm, amscd, amssymb, amsfonts, amsbsy}
\usepackage[usenames, dvipsnames]{color}
\usepackage{enumerate}
\usepackage{mathrsfs}
\usepackage{verbatim}

\numberwithin{equation}{section}

\theoremstyle{plain}
\newtheorem{theorem}{Theorem}[section]
\newtheorem{lemma}[theorem]{Lemma}
\newtheorem{corollary}[theorem]{Corollary}

\theoremstyle{definition}
\newtheorem{definition}[theorem]{Definition}

\theoremstyle{remark}
\newtheorem{remark}[theorem]{Remark}
\newtheorem{notation}{Notation}[section]

\makeatletter
\def\dashint{\operatorname%
{\,\,\text{\bf--}\kern-.98em\DOTSI\intop\ilimits@\!\!}}
\makeatother

\def\bR{\mathbb{R}}
\def\cA{\mathcal{A}}
\def\cL{\mathcal{L}}
\def\cG{\mathcal{G}}
\def\cF{\mathcal{F}}
\def\cU{\mathcal{U}}

\begin{document}
\title[Stokes system]{Gradient estimates for Stokes systems in domains}

\author[J. Choi]{Jongkeun Choi}
\address[J. Choi]{Division of Applied Mathematics, Brown University, 182 George Street, Providence, RI 02912, USA}
\email{Jongkeun\_Choi@brown.edu}

\thanks{J. Choi was supported by Basic Science Research Program
through the National Research Foundation of Korea (NRF) funded by the Ministry of Education
(2017R1A6A3A03005168)}

\author[H. Dong]{Hongjie Dong}
\address[H. Dong]{Division of Applied Mathematics, Brown University, 182 George Street, Providence, RI 02912, USA}

\email{Hongjie\_Dong@brown.edu}

\thanks{H. Dong was partially supported by the NSF under agreement DMS-1600593.}


\begin{abstract}
We study the stationary Stokes system with Dini mean oscillation coefficients in a domain having $C^{1,\rm{Dini}}$ boundary.
We prove that if $(u, p)$ is a weak solution of the system with zero Dirichlet boundary condition, then $(Du,p)$ is continuous up to the boundary.
We also prove a weak type-$(1,1)$ estimate for $(Du, p)$.
\end{abstract}

\maketitle

\section{Introduction and main results}		\label{Sec1}

We consider the stationary Stokes system with variable coefficients
\begin{equation}		\label{171230@A1}
\left\{
\begin{aligned}
\cL u+\nabla p=D_\alpha f_\alpha &\quad \text{in }\, \Omega,\\
\operatorname{div}u=g &\quad \text{in }\, \Omega,
\end{aligned}
\right.
\end{equation}
where $\Omega$ is a bounded domain in $\bR^d$, $d\ge 2$.
The differential operator $\cL$ is in divergence form acting on column vector-valued functions $u=(u^1,\ldots,u^d)^\top$ as follows:
$$
\cL u=D_\alpha (A^{\alpha\beta}D_\beta u),
$$
where the coefficients $ A^{\alpha\beta}= A^{\alpha\beta}(x)$ are $d\times d$ matrix-valued functions on $\Omega$, which satisfy the strong ellipticity condition,
i.e., there is a constant $\lambda\in (0,1]$ such that for any $x\in \bR^d$ and $ \xi_\alpha \in \bR^d$, $\alpha\in \{1,\ldots,d\}$, we have
$$
|A^{\alpha\beta}(x)|\le \lambda^{-1}, \quad \sum_{\alpha,\beta=1}^dA^{\alpha\beta}(x)\xi_\beta\cdot \xi_\alpha\ge \lambda \sum_{\alpha=1}^d|\xi_\alpha|^2.
$$

In a recent paper \cite{arXiv:1803.05560}, we investigated minimal regularity assumptions on coefficients and data for $W^{1,\infty}$ and $C^1$ regularity of weak solutions to the Stokes system
in a ball and a half ball.
One of the results in \cite{arXiv:1803.05560} is that every weak solutions of \eqref{171230@A1} satisfy
$$
(u,p)\in C^1(\Omega')^d\times C(\Omega'), \quad \Omega' \Subset \Omega
$$
provided that the coefficients and data are of {\emph{Dini mean oscillation}}.
We say that a function is of Dini mean oscillation if its $L^1$-mean oscillation satisfies the Dini condition; see Definition \ref{D2} for more precise definition.
This class of functions was first introduced by Dong-Kim in \cite{MR3620893} for $C^1$ and $C^2$ regularity of solutions to elliptic equations in divergence and nondivergence form.
A local weak type-$(1,1)$ estimate for $(Du, p)$ was also proved in \cite{arXiv:1803.05560}.

In this paper, we extend the aforementioned results in \cite{arXiv:1803.05560} to domains up to the boundary.
More precisely, we prove that weak solutions of the Stokes system \eqref{171230@A1} with zero Dirichlet boundary condition satisfy
\begin{equation}		\label{180103@A1}
(u,p)\in C^1(\overline{\Omega})^d\times C(\overline{\Omega})
\end{equation}
provided that the coefficients and data are of Dini mean oscillation, and that $\Omega$ has $C^{1,\rm{Dini}}$ boundary.
As an application, we obtain Schauder estimate and regularity for weak solutions, which were studied in \cite[Theorem 1.3, p. 198]{MR0641818}.
We also prove the global weak type-$(1,1)$ estimate for $(Du, p)$ under a stronger assumption on the coefficients and the boundary.

Our argument in establishing \eqref{180103@A1} is based on the approach used in \cite{MR3747493}, where the authors proved boundary $C^1$-estimates for divergence type elliptic equations
$$
D_i (a^{ij}D_j u)=\operatorname{div}f
$$
with Dini mean oscillation coefficients on a domain having $C^{1,\rm{Dini}}$ boundary.
The key ingredient is $L^q$-mean oscillation estimates with $q\in (0,1)$ for derivatives of solutions on the boundary.
In \cite{MR3747493}, such mean oscillation estimates were obtained near a flat boundary and then the boundary $C^1$-estimate follows from that on the half ball since the mapping of flattening boundary preserves the regularity assumptions on the coefficients and data.
However, this argument does not work for the Stokes system because after the mapping the pressure term and the divergence equation give rise to extra terms which are {\em not} of Dini mean oscillation.
In this paper, we establish the $L^q$-mean oscillation estimate near curved boundary.
To this end, we fix a point  $x_0=(x_{01},x_0')\in \partial \Omega$  and a coordinate system so that the $C^{1,\rm{Dini}}$ function $\chi$ defining $\partial\Omega$ near $x_0$ satisfies $|\nabla_{x'}\chi(x_0')|=0$.
Then, in this coordinate system, we employ the mapping of flattening boundary to control the $L^q$-mean oscillation at $x_0$.
Therefore, our mean oscillation estimate at the boundary point $x_0$ depends on the coordinate system and the $C^{1,\rm{Dini}}$ function $\chi$ associated with $x_0$; see Lemma \ref{171101@lem1}.
This makes the arguments much more involved.

The remainder of the paper is organized as follows.
In the rest of this section, we state our main results along with some definitions and assumptions.
In Section \ref{Sec2}, we provide the proofs of the main theorems.
In Appendix, we provide the proofs of some lemmas used in the paper.

For any $x\in \overline{\Omega}$ and $r>0$, we denote $\Omega_r(x)=\Omega\cap B_r(x)$, where $B_r(x)$ is a usual Euclidean ball of radius $r$ centered at $x$.
We denote $B_r^+(x)=B_r(x)\cap \bR^d_+$, where
$$
\bR^d_+=\{x=(x_1,x')\in \bR^d:x_1>0, \, x'\in \bR^{d-1}\}.
$$
For $0 < q \le \infty$, let $L^q(\Omega)$ be the space consisting of measurable functions on $\Omega$ that are $q$-th
integrable.
We define
$$
\tilde{L}^q(\Omega)=\{f\in L^q(\Omega): (f)_\Omega=0\},
$$
where $(f)_\Omega$ is the average of $f$ over $\Omega$, i.e.,
$$
(f)_\Omega=\dashint_\Omega f\,dx=\frac{1}{|\Omega|}\int_\Omega f\,dx.
$$
For $1\le q\le \infty$, we denote by $W^{1,q}(\Omega)$ the usual Sobolev space and by $W^{1,q}_0(\Omega)$ the completion of $C^\infty_0(\Omega)$ in $W^{1,q}(\Omega)$.
We define the H\"older semi-norm by
$$
[f]_{C^{\gamma}(\Omega)}:=\sup_{\substack{x,y\in \Omega \\ x\neq y}} \frac{|f(x)-f(y)|}{|x-y|^\gamma}, \quad 0<\gamma<1.
$$
We say that a measurable function $\omega:(0,a]\to [0,\infty)$ is a Dini function
provided that there are constants $c_1, c_2>0$ such that
\begin{equation}		\label{171006@eq1}
c_1\omega(t)\le \omega(s)\le c_2\omega(t)		\quad \text{whenever }\, \frac{t}{2}\le s\le t\le a
\end{equation}
and  that $\omega$ satisfies the Dini condition
\begin{equation}		\label{180315@A1}
\int_0^{a} \frac{\omega (t)}{t} \,dt<\infty.
\end{equation}

\begin{definition}		\label{D2}
Let $f\in L^1(\Omega)$.
\begin{enumerate}[$(i)$]
\item
We say that $f$ is {\em{uniformly Dini continuous}} in $\Omega$ if the function $\varrho_{f}:(0,1] \to [0,\infty)$ defined by
$$
\varrho_{f}(r):=\sup_{x_0\in \Omega} \sup_{x,y\in \Omega_r(x_0)}|f(x)-f(y)|
$$
is a Dini function.
\item
We say that $f$ is of {\em{Dini mean oscillation}} in $\Omega$ if the function $\omega_{f}:(0, 1]\to [0,\infty)$ defined by
$$
\omega_{f}(r):=\sup_{x\in \overline{\Omega}}\dashint_{\Omega_r(x)} \big|f-(f)_{\Omega_r(x)}\big|\,dy
$$
satisfies the Dini condition
$$
\int_0^{1} \frac{\omega_{f}(t)}{t}\,dt<\infty.
$$

\end{enumerate}
\end{definition}

\begin{remark}		\label{171020@rmk1}
Assume that $|\Omega_r(x)|\ge A_0 r^d$ for all $x\in \overline{\Omega}$ and $0<r\le 1$.
If $f$ is of Dini mean oscillation in $\Omega$, then $f$ is uniformly continuous in $\Omega$ with its modulus of continuity controlled by $\omega_f$.
Moreover, since $\omega_{f}$ satisfies the condition \eqref{171006@eq1} with $(c_1,c_2)=(c_1,c_2)(d,A_0)$ (see, for instance, \cite[p. 495]{MR3615500}), we have that $\omega_{f}:(0,1]\to [0,\infty)$ is a Dini function.
\end{remark}

\begin{definition}		\label{D3}
Let $\Omega$ be a domain in $\bR^d$.
We say that $\Omega$ has $C^{1, \rm{Dini}}$ boundary if there exist a constant $R_0\in (0,1]$ and a Dini function $\varrho_0:(0, 1]\to [0,\infty)$ such that the following holds:
For any $x_0=(x_{01},x_0')\in \partial \Omega$, there exists a $C^{1,\rm{Dini}}$ function (i.e., $C^1$ function whose derivatives are uniformly Dini continuous) $\chi:\bR^{d-1}\to \bR$ and a coordinate system depending on $x_0$ such that
\begin{equation}		\label{171101@E1}
\varrho_{\nabla_{x'}\chi}(r)\le \varrho_0(r) \quad \text{for all }\, r\in (0,R_0),
\end{equation}
and that in the new coordinate system, we have
\begin{equation}		\label{171101@E2}
|\nabla_{x'}\chi(x_0')|=0, \quad \Omega_{R_0}(x_0)=\{x\in B_{R_0}(x_0): x_1>\chi(x')\}.
\end{equation}
\end{definition}

Now, we state our main theorems.

\begin{theorem}		\label{M4}
Let $\Omega$ be a bounded domain in $\bR^d$ having $C^{1,\rm{Dini}}$ boundary.
Assume that $(u,p)\in W^{1,2}_0(\Omega)^d\times \tilde{L}^2(\Omega)$ is the weak solution of
\begin{equation}		\label{171006@eq2}
\left\{
\begin{aligned}
\cL u+\nabla p=D_\alpha f_\alpha \quad \text{in }\, \Omega,\\
\operatorname{div} u=g-(g)_\Omega \quad \text{in }\, \Omega,
\end{aligned}
\right.
\end{equation}
where $f_\alpha\in L^{2}(\Omega)^d$ and $g\in L^2(\Omega)$.
\begin{enumerate}[$(a)$]
\item
If $A^{\alpha\beta}$, $f_\alpha$, and $g$ are of Dini mean oscillation in $\Omega$, then we have
$$
(u,p)\in C^1(\overline{\Omega})^d\times C(\overline{\Omega}).
$$
\item
Let $0<\gamma_0<1$ and $\partial \Omega$ be $C^{1,\gamma_0}$, i.e., $\varrho_0 (r) = N r^{\gamma_0}$ for some constant $N>0$.
If it holds that $[A^{\alpha\beta}]_{C^{\gamma_0}(\Omega)}+[f_\alpha]_{C^{\gamma_0}(\Omega)}+[g]_{C^{\gamma_0}(\Omega)}<\infty$,
then we have
$$
(u,p)\in C^{1,\gamma_0}(\overline{\Omega})^d\times C^{\gamma_0}(\overline{\Omega}).
$$
\end{enumerate}
\end{theorem}

\begin{remark}		
By the same reasoning as \cite[Remark 2.4]{arXiv:1803.05560},
one can extend the results in Theorem \ref{M4} to the solution of the system
$$
\left\{
\begin{aligned}
\cL u+\nabla p=f+D_\alpha f_\alpha \quad \text{in }\, \Omega,\\
\operatorname{div} u=g-(g)_\Omega \quad \text{in }\, \Omega,
\end{aligned}
\right.
$$
where $f\in L^q(\Omega)^d$ with $q>d$.
\end{remark}

In the next theorem, we prove the global weak type-$(1,1)$ estimate for $Du$ and $p$.

\begin{theorem}		\label{M5}
Let $\Omega$ be a bounded domain in $\bR^d$ having $C^{1,\rm{Dini}}$ boundary.
Assume that $(u, p)\in W^{1,q}_0(\Omega)^d\times \tilde{L}^q(\Omega)$ is the weak solution of \eqref{171006@eq2}, where $f_\alpha\in L^q(\Omega)^d$, $g\in L^q(\Omega)$, and $q\in (1,\infty)$.
If $A^{\alpha\beta}$ are of Dini mean oscillation in $\Omega$ and
\begin{equation}		\label{171127@B1}
\varrho_0(r)+\omega_{A^{\alpha\beta}}(r)\le C_0 (\ln r)^{-2}, \quad \forall r\in (0,1/2),
\end{equation}
then for any $t>0$, we have
\begin{equation}		\label{180315@eq2}
\big|\{x\in \Omega:|Du(x)|+|p(x)|>t\}\big|\le \frac{C}{t}\int_\Omega (|f_\alpha|+|g|)\,dx,
\end{equation}
where the constant $C$ depends only on $d$, $\lambda$, $\Omega$, $R_0$, $\varrho_0$, $\omega_{A^{\alpha\beta}}$,  and $C_0$.
\end{theorem}

\begin{remark}		\label{180419@rmk1}
Under the hypothesis of Theorem \ref{M4} $(a)$, the unique solvability of the problem \eqref{171006@eq2} is available in the solution space $W^{1,2}_0(\Omega)^d\times \tilde{L}^2(\Omega)$ as well as $W^{1,q}_0(\Omega)^d\times \tilde{L}^q(\Omega)$ with $q\in (1,\infty)$, when $f_\alpha\in L^{q}(\Omega)^d$ and $g\in L^q(\Omega)$; see the proof of Theorem \ref{M5}.
Therefore, in Theorems \ref{M4} and \ref{M5}, the weak solutions indeed exist.
\end{remark}

We present the $W^{1,q}$-estimate for a $W^{1,1}$-weak solution, which follows from Theorem \ref{M4}, the solvability result mentioned in Remark \ref{180419@rmk1}, and the argument in Brezis \cite{MR2465684} (see also \cite[Appendix]{MR2548032}).
For a proof, one may refer to the proofs of \cite[Theorems 2.5 and 5.4]{arXiv:1803.05560}, where we proved the $W^{1,q}$-estimates for $W^{1,1}$-weak solutions to the Stokes system with partially Dini mean oscillation coefficients in a ball and a half ball.

\begin{corollary}		\label{180419@cor1}
Let $q\in (1,\infty)$ and $\Omega$ be a bounded domain in $\bR^d$ having $C^{1,\rm{Dini}}$ boundary.
Assume that $(u, p)\in W^{1,1}_0(\Omega)^d\times \tilde{L}^1(\Omega)$ is a weak solution of \eqref{171006@eq2}, where $f_\alpha\in L^q(\Omega)^d$ and $g\in L^q(\Omega)$.
If $A^{\alpha\beta}$ are of  Dini mean oscillation in $\Omega$, then we have $(u, p)\in W^{1,q}_0(\Omega)^d\times \tilde{L}^q(\Omega)$ with the estimate
$$
\|u\|_{W^{1,q}(\Omega)}+\|p\|_{L^q(\Omega)}\le C\big(\|u\|_{W^{1,1}(\Omega)}+\|p\|_{L^1(\Omega)}+\|f_\alpha\|_{L^q(\Omega)}+\|g\|_{L^q(\Omega)}\big),
$$
where the constant $C$ depends only on $d$, $\lambda$, $\Omega$, $R_0$, $\varrho_0$, and $\omega_{A^{\alpha\beta}}$.
\end{corollary}

We finish this section with a remark that, by Corollary \ref{180419@cor1} the results in Theorems \ref{M4} and \ref{M5} still hold under the assumption that $(u, p)\in W^{1,1}_0(\Omega)^d\times \tilde{L}^1(\Omega)$.

\section{Proof of main Theorems}		\label{Sec2}

Hereafter in the paper,  we use the following notation.
\begin{notation}
For nonnegative (variable) quantities $A$ and $B$,
we denote $A\lesssim B$ if there exists a generic positive constant C such that $A \le CB$.
We add subscript letters like $A\lesssim_{a,b} B$ to indicate the dependence of the implicit constant $C$ on the parameters $a$ and $b$.
\end{notation}

\subsection{Proof of Theorem \ref{M4}}
We shall derive a priori estimates for $Du$ and $p$ by assuming that $(u,p)\in C^1(\overline{\Omega})^d\times C(\overline{\Omega})$.
The general case follows from a standard approximation argument.

Throughout this proof, we use the following notation and properties.
Recall that $\varrho_0$ is the Dini function from Definition \ref{D3}.

\begin{enumerate}[i.]
\item
We set $q=1/2$ and
$$
\Phi(x_0, r):=\inf_{\substack{\theta\in \bR \\ \Theta\in \bR^{d\times d}}}\bigg(\dashint_{\Omega_r(x_0)}|Du-\Theta|^q+|p-\theta|^q\,dx\bigg)^{1/q}.
$$
\item
For any $x\in \overline{\Omega}$ and $r\in (0,1]$, we have
\begin{equation}		\label{171127@eq3}
r^d\lesssim_{d,R_0, \varrho_0} |\Omega_r(x)|.
\end{equation}
\item
For $\gamma\in (0,1)$ and $\kappa\in (0,1/2]$, we define
$$
\tilde{\varrho}_{0}(r):=\varrho_0(r)+\sum_{i=1}^\infty \kappa^{\gamma i}\big(\varrho_{0}(\kappa^{-i}r)[\kappa^{-i} r<1]+\varrho_{0}(1)[\kappa^{-i}r\ge 1]\big),
$$
where we use Inverse bracket notation;
i.e., $[P]=1$ if $P$ is true and $[P]=0$ otherwise.
By Lemma \ref{171024@lem1}, $\tilde{\varrho}_0:(0, 1]\to [0,\infty)$ is a Dini function satisfying
\begin{equation}		\label{171102@eq3}
\tilde{\varrho}_0(t)\lesssim_{\varrho_0} \tilde{\varrho}_0(s)\lesssim_{\varrho_0}\tilde{\varrho}_0(t) \quad \text{whenever }\, \frac{t}{2}\le s\le t\le 1.
\end{equation}
Moreover, by the comparison principle for Riemann integrals, we have
$$
\sum_{j=0}^\infty \tilde{\varrho}_0(\kappa^j r)\lesssim_{\varrho_0,\kappa} \int_0^r \frac{\tilde{\varrho}_0(t)}{t}\,dt<\infty
$$
for all $r\in (0,1]$.
\item
For $\gamma\in (0,1)$, $\kappa\in (0,1/2]$, and $f\in L^1(\Omega)$, we denote
$$
\tilde{\omega}_{f}(r):=\sum_{i=1}^\infty \kappa^{\gamma i}\big(\omega_{f}(\kappa^{-i}r)[\kappa^{-i} r<1]+\omega_{f}(1)[\kappa^{-i}r\ge 1]\big).
$$
By Remark \ref{171020@rmk1}, \eqref{171127@eq3}, and Lemma \ref{171024@lem1}, if $f$ is of Dini mean oscillation in $\Omega$, then $\tilde{\omega}_{f}:(0, 1]\to [0,\infty)$ is a Dini function satisfying
$$
\tilde{\omega}_{f}(t)\lesssim_{d,R_0, \varrho_0} \tilde{\omega}_{f}(s)\lesssim_{d,R_0, \varrho_0} \tilde{\omega}_{f}(t) \quad \text{whenever }\,\frac{t}{2}\le s\le t\le 1.
$$
Moreover,  we have
\begin{equation}		\label{171229@eq1a}
\sum_{j=0}^\infty \tilde{\omega}_{f}(\kappa^j r)\lesssim_{d ,R_0, \varrho_0, \kappa}  \int_0^{r} \frac{\tilde{\omega}_{f}(t)}{t}\,dt <\infty
\end{equation}
for all $r\in (0,1]$.
\end{enumerate}

To prove Theorem \ref{M4}, we will use the following three lemmas related to $L^q$-mean oscillation estimates for $Du$ and $p$.
The first lemma is about the interior estimates, which is an adaptation of \cite[Lemma 4.3]{arXiv:1803.05560}.

\begin{lemma}		\label{171228@lem1}
Let $x_0\in \Omega$ and $\gamma\in (0,1)$.
Under the same hypothesis of Theorem \ref{M4} $(a)$,
there exists a constant $\kappa_1\in (0,1/2]$ depending only on $d$, $\lambda$, and $\gamma$, such that the following hold.
\begin{enumerate}[$(i)$]
\item
For any $0<\kappa\le \kappa_1$ and $0<r\le \min\{1,\operatorname{dist}(x_0, \partial \Omega)/4\}$, we have
$$
\begin{aligned}
\sum_{j=0}^\infty \Phi(x_0, \kappa^j r)&\lesssim_{d,\lambda,\gamma,R_0, \varrho_0,\kappa} \Phi(x_0, r)\\
&\quad +\|Du\|_{L^\infty(B_r(x_0))}\int_0^r \frac{\tilde{\omega}_{A^{\alpha\beta}}(t)}{t}\,dt+\int_0^r \frac{\tilde{\omega}_{f_\alpha}(t)+\tilde{\omega}_{g}(t)}{t}\,dt.
\end{aligned}
$$
\item
For any $0<\kappa\le \kappa_1$ and $0<\rho\le r\le \min\{1,\operatorname{dist}(x_0, \partial \Omega)/4\}$, we have
$$
\Phi(x_0, \rho) \lesssim_{d,\lambda,\gamma,\kappa} \left(\frac{\rho}{r}\right)^\gamma \Phi(x_0, r)+ \|Du\|_{L^\infty(B_r(x_0))}\tilde{\omega}_{A^{\alpha\beta}}(\rho)+\tilde{\omega}_{f_\alpha}(\rho)+\tilde{\omega}_g (\rho).
$$
\end{enumerate}
\end{lemma}

\begin{proof}
By following the proof of \cite[Lemma 4.3]{arXiv:1803.05560}, we see that
$$
\begin{aligned}
\Phi(x_0, \kappa r) &\le C_0 \kappa \Phi(x_0, r)\\
&\quad +C_0\kappa^{-d/q}\big(\|Du\|_{L^\infty (B_r(x_0))}\omega_{A^{\alpha\beta}}(r)+\omega_{f_\alpha}(r)+\omega_g(r)\big)
\end{aligned}
$$
for all $0<\kappa\le 1/2$ and $0<r\le\min\{1,\operatorname{dist}(x_0, \partial \Omega)/4\}$, where $C_0=C_0(d,\lambda)>0$.
We take $\kappa_1=\kappa_1(d,\lambda,\gamma)\in (0, 1/2]$ such that $C_0 \kappa_1^{1-\gamma}\le 1$.
Then for any $0<\kappa\le \kappa_1$, we have
$$
\Phi(x_0, \kappa r) \le \kappa^\gamma \Phi(x_0, r)+C\big( \|Du\|_{L^\infty (B_r(x_0))}\omega_{A^{\alpha\beta}}(r) + \omega_{f_\alpha}(r)+\omega_g(r)\big),
$$
where $C=C(d,\lambda,\gamma,\kappa)$.
By iterating, we obtain for $j\in \{1,2,\ldots\}$ that
\begin{equation}		\label{171229@eq2}
\begin{aligned}
\Phi(x_0, \kappa^j r)&\le \kappa^{\gamma j} \Phi(x_0, r)\\
&\quad +C\big(\|Du\|_{L^\infty(B_r(x_0))}\tilde{\omega}_{A^{\alpha\beta}}(\kappa^j r)+\tilde{\omega}_{f_\alpha}(\kappa^j r)+\tilde{\omega}_g(\kappa^j r)\big),
\end{aligned}
\end{equation}
where we used the fact that
\begin{equation}		\label{171229@eq3}
\sum_{i=1}^j \kappa^{\gamma(i-1)}\omega_{\bullet}(\kappa^{j-i}r)\le \kappa^{-\gamma} \tilde{\omega}_{\bullet}(\kappa^j r).
\end{equation}
Taking the summations of both sides of \eqref{171229@eq2} with respect to $j=0,1,2,\ldots$, and using \eqref{171229@eq1a}, we see that the assertion $(i)$ holds.

For given $\rho\in (0, r]$, let $j$ be an integer such that
$$
\kappa^{j+1}<\frac{\rho}{r}\le \kappa^j.
$$
If $j=0$, then obviously we have
$$
\Phi(x_0, \rho)\lesssim_{d,\kappa} \Phi(x_0, r)\lesssim_{d,\kappa,\gamma} \left(\frac{\rho}{r}\right)^{\gamma}\Phi(x_0,r).
$$
On the other hand, if $j\ge 1$, then by \eqref{171229@eq2} with $\rho$ in place of $\kappa^j r$, we get
$$
\begin{aligned}
\Phi(x_0, \rho)&\lesssim \kappa^{\gamma j}\Phi(x_0, \kappa^{-j}\rho)+\|Du\|_{L^\infty(B_{\kappa^{-j}\rho}(x_0))}\tilde{\omega}_{A^{\alpha\beta}}(\rho)+\tilde{\omega}_{f_\alpha}(\rho)+\tilde{\omega}_g(\rho)\\
&\lesssim \left(\frac{\rho}{r}\right)^\gamma\Phi(x_0, r)+\|Du\|_{L^\infty(B_r(x_0))}\tilde{\omega}_{A^{\alpha\beta}}(\rho)+\tilde{\omega}_{f_\alpha}(\rho)+\tilde{\omega}_g(\rho).
\end{aligned}
$$
Therefore, the assertion $(ii)$ holds.
The lemma is proved.
\end{proof}

In the next lemma, we prove $L^q$-mean oscillation estimates of linear combinations of $Du$ and $p$ at $x_0\in \partial \Omega$.
We note that the $L^q$-mean oscillation and its estimates depend on the coordinate system associated with $x_0$.

\begin{lemma}		\label{171101@lem1}
Let $x_0\in \partial \Omega$ and $\gamma\in (0,1)$.
Let us fix a $C^{1,\rm{Dini}}$ function $\chi:\bR^{d-1}\to \bR$ and a coordinate system associated with $x_0$ satisfying \eqref{171101@E1} and \eqref{171101@E2} in Definition \ref{D3}.
In this coordinate system, we define
$$
\Psi(x_0, r):=\inf_{\substack{\theta\in \bR \\  \Theta\in \bR^{d}}}\bigg(\dashint_{\Omega_r(x_0)}|D_1 u-\Theta|^q +\sum_{i=2}^d|D_i \chi D_1 u+D_i u|^q+|p-\theta|^q\,dx\bigg)^{1/q}.
$$
Then under the same hypothesis of Theorem \ref{M4} $(a)$,
there exist constants
$$
R_1=R_1(\varrho_0, R_0)\in (0, R_0/4) \quad \text{and}\quad \kappa_{2}=\kappa_{2}(d,\lambda,\gamma,  R_0,\varrho_0)\in (0,1/8]
$$
such that the following hold.
\begin{enumerate}[$(i)$]
\item
For any $0<\kappa\le \kappa_2$ and $0<r\le 2R_1$, we have
$$
\begin{aligned}
\sum_{j=0}^\infty \Psi(x_0, \kappa^j r)
&\lesssim_{d,\lambda,\gamma, R_0,\varrho_0 ,\kappa} \Psi(x_0, r)\\
&\quad + \big(\|Du\|_{L^\infty(\Omega_r(x_0))}+\|p\|_{L^\infty(\Omega_r(x_0))}\big)\int_0^r \frac{\tilde{\varrho}_0(t)+\tilde{\omega}_{A^{\alpha\beta}}(t)}{t}\,dt\\
&\quad +\|f_\alpha\|_{L^\infty(\Omega_r(x_0))}\int_0^r \frac{\tilde{\varrho}_0(t)}{t}\,dt +\int_0^r \frac{\tilde{\omega}_{f_\alpha}(t)+\tilde{\omega}_g(t)}{t}\,dt.
\end{aligned}
$$
\item
For any $0<\kappa\le \kappa_2$ and $0<\rho\le r\le 2R_1$, we have
$$
\begin{aligned}
\Psi(x_0,\rho)
&\lesssim_{d,\lambda,\gamma, R_0,\varrho_0,\kappa} \left(\frac{\rho}{r}\right)^\gamma \Psi(x_0, r)\\
&\quad +\big(\|Du\|_{L^\infty(\Omega_r(x_0))}+\|p\|_{L^\infty(\Omega_r(x_0))}\big) \big(\tilde{\varrho}_0(\rho)+\tilde{\omega}_{A^{\alpha\beta}}(\rho)\big)\\
&\quad +\|f_\alpha\|_{L^\infty(\Omega_r(x_0))}\tilde{\varrho}_0(\rho)+\tilde{\omega}_{f_\alpha}(\rho)+\tilde{\omega}_g(\rho).
\end{aligned}
$$
\end{enumerate}
\end{lemma}

\begin{proof}
Recall that we use $0=(0,0')$, $x=(x_1,x')$, and $y=(y_1,y')$ to denote points in $\bR^d$.
Without loss of generality, we assume that $x_0=0\in \partial \Omega$ and $\chi(0')=0$.
We denote $B_R=B_R(0)$, $B_R^+=B_R^+(0)$, and $\Omega_R=\Omega_R(0)$.
Since $|\nabla_{x'}\chi(0')|=0$, it follows from \eqref{171101@E1} that there exists a constant $R_1=R_1(\varrho_0,R_0)\in (0, R_0)$ satisfying
\begin{equation}		\label{171101@eq4}
|\nabla_{x'}\chi(x')|\le 1/2 \quad \text{if }\, |x'|\le R_1.
\end{equation}
Let $\Gamma(y)=(y_1+\chi(y'),y')$ and $\Lambda(x)=\Gamma^{-1}(x)=(x_1-\chi(x'),x')$.
We divide the proof into several steps.

{\em{Step 1}.}
In this step, we prove that
\begin{equation}	\label{171101@D2}	
B_{R_1/2}^+\subset \Lambda (\Omega_{R_1}),
\end{equation}
\begin{equation}		\label{171101@D2a}
\Omega_{r/2}\subset \Gamma(B_{r}^+) \subset \Omega_{2r} \quad \text{for }\, r\in (0, R_1/2].
\end{equation}
To prove \eqref{171101@D2}, assume that $y\in B_{R_1/2}^+$.
Then we have
$$
\begin{aligned}
|y_1+\chi(y')|^2+|y'|^2&\le 2|y_1|^2+2|\chi(y')|^2+|y'|^2\\
&\le |y|^2+|y_1|^2+2|\chi(y')|^2\\
&< \frac{R_1^2}{2} +2|\chi(y')|^2.
\end{aligned}
$$
Notice from \eqref{171101@eq4} that
$$
|\chi(y')|^2=|\chi(y')-\chi(0')|^2\le \frac{|y'|^2}{4}<\frac{R_1^2}{4}.
$$
Combining the above two inequalities, we have $|y_1+\chi(y')|^2+|y'|^2< R_1^2$, which implies that $y\in \Lambda(\Omega_{R_1})$.
Thus we get \eqref{171101@D2}.
Using a similar argument, we have \eqref{171101@D2a}.

{\em{Step 2}.}
In this step, we use the standard technique of flattening the boundary.
We denote
$$
v(y)=u(\Gamma(y)), \quad \pi(y)=p(\Gamma(y)), \quad   b(y)=(0, D_2 \chi(y'),\ldots,D_d\chi(y'))^{\top}.
$$
Since $(u,p)$ satisfies \eqref{171006@eq2},
we have that
$$
\left\{
\begin{aligned}
D_\alpha(\cA^{\alpha\beta}D_\beta v)+\nabla \pi=D_\alpha F_\alpha +D_1 (\pi b) &\quad \text{in }\, B_{R_1}^+,\\
\operatorname{div} v =G + D_1v \cdot b &\quad \text{in }\, B_{R_1}^+,\\
v=0 &\quad \text{on }\, B_{R_1}\cap \partial \bR^d_+,
\end{aligned}
\right.
$$
where we set
$$
\cA^{\alpha\beta}=D_\ell \Lambda^\beta D_k \Lambda^\alpha {A}^{k\ell}(\Gamma), \quad F_\alpha=D_k \Lambda^\alpha {f}_k(\Gamma), \quad G=g(\Gamma)-(g)_\Omega.
$$
Let $0<r\le R_1/4$.
For a given function $f$, we denote $ \overline{f}=(f)_{B_r^+}$.
Define an elliptic operator $\cL_0$ by
$$
\cL_0 v=D_\alpha(\overline{\cA^{\alpha\beta}}D_\beta v),
$$
and observe that $(v,\pi)$ satisfies
$$
\left\{
\begin{aligned}
\cL_0 v+\nabla \pi=D_\alpha \cF_\alpha &\quad \text{in }\, B_{R_1}^+,\\
\operatorname{div} v =\cG + \overline{G} &\quad \text{in }\, B_{R_1}^+,\\
v=0 &\quad \text{on }\, B_{R_1}\cap \partial \bR^d_+,
\end{aligned}
\right.
$$
where
$$
\cF_\alpha=\big(\overline{\cA^{\alpha\beta}}-\cA^{\alpha\beta}\big)D_\beta v+F_\alpha-\overline{F_\alpha} + \delta_{1\alpha} \pi b, \quad \cG=G-\overline{G} +D_1 v\cdot b.
$$
Here, $\delta_{ij}$ is the usual Kronecker delta symbol.
We decompose
\begin{equation}		\label{171101@D1a}
(v,\pi)=(v_1,\pi_1)+(v_2,\pi_2),
\end{equation}
where $(v_1,\pi_1)\in W^{1,2}_0(B_{4r}^+)^d\times \tilde{L}^2(B_{4r}^+)$ is the weak solution of the problem
$$
\left\{
\begin{aligned}
\cL_0 v_1+\nabla \pi_1=D_\alpha(I_{B_r^+}\cF_\alpha) &\quad \text{in }\, B_{4r}^+,\\
\operatorname{div} v_1=I_{B_r^+}\cG-\big(I_{B_r^+}\cG\big)_{B_{4r}^+} &\quad \text{in }\, B_{4r}^+.
\end{aligned}
\right.
$$
Here, $I_{B_r^+}$ is the characteristic function.
By \cite[Lemma 6.5]{arXiv:1803.05560} with scaling, we have for $t>0$ that
$$
\big|\{y\in B_r^+:|Dv_1(y)|+|\pi_1(y)|>t\}\big|\lesssim_{d,\lambda} \frac{1}{t}\int_{B_r^+} ( |\cF_\alpha|+|\cG|)\,dy.
$$
This inequality implies that for $\tau>0$,
$$
\begin{aligned}
&\int_{B_r^+}(|Dv_1|+|\pi_1|)^q\,dy\\
&=\bigg(\int_0^\tau+\int_\tau^\infty \bigg)q t^{q-1} \big|\{y\in B_r^+ : |Dv_1(y)|+|\pi_1(y)|>t\}\big|\,dt\\
&\lesssim |B_r^+|\tau^q +\bigg(\int_{B_r^+}|\cF_\alpha|+|\cG|\,dy\bigg)\tau^{q-1}.
\end{aligned}
$$
By optimizing over $\tau$ and taking the $q$-th root, we have
\begin{equation}		\label{171101@D1}
\bigg(\dashint_{B_r^+} (|Dv_1|+|\pi_1|)^q\,dy\bigg)^{1/q}\lesssim \dashint_{B_r^+} (|\cF_\alpha|+|\cG|)\,dy.
\end{equation}
Since $(v_2,\pi_2)=(v,\pi)-(v_1,\pi_1)$ satisfies
$$
\left\{
\begin{aligned}
\cL_0 v_2+\nabla \pi_2=0 &\quad \text{in }\, B_r^+,\\
\operatorname{div} v_2=\big(I_{B_r^+}\cG\big)_{B_{4r}^+}+\overline{G} &\quad \text{in }\, B^+_r,\\
v_2=0 &\quad \text{on }\, B_r\cap \partial \bR^d_+,
\end{aligned}
\right.
$$
by \cite[Lemma 6.3]{arXiv:1803.05560}, we have for any $\kappa\in (0,1/2]$,
\begin{equation}		\label{171101@D1b}
\begin{aligned}
&\bigg(\dashint_{B_{\kappa r}^+}\big|D_1v_2-(D_1v_2)_{B_{\kappa r}^+}\big|^q+|D_{y'}v_2|^q+\big|\pi_2-(\pi_2)_{B_{\kappa r}^+}\big|^q\,dy\bigg)^{1/q}\\
&\lesssim_{d,\lambda} \kappa \inf_{\Theta\in \bR^d}\bigg(\dashint_{B_r^+} |D_1v_2-\Theta |^q+|D_{y'}v_2|^q\,dy\bigg)^{1/q}.
\end{aligned}
\end{equation}
Observe from \eqref{171101@D1a} that
$$
\begin{aligned}
&\bigg(\dashint_{B_{\kappa r}^+} \big|D_1v-(D_1v_2)_{B_{\kappa r}^+}\big|^q+|D_{y'}v|^q+\big|\pi-(\pi_2)_{B_{\kappa r}^+}\big|^q\,dy\bigg)^{1/q}\\
&\lesssim \bigg(\dashint_{B_{\kappa r}^+} \big|D_1v_2-(D_1v_2)_{B_{\kappa r}^+}\big|^q+|D_{y'}v_2|^q +\big|\pi_2-(\pi_2)_{B_{\kappa r}^+}\big|^q\,dy \bigg)^{1/q}\\
&\quad +\bigg(\dashint_{B_{\kappa r}^+} |Dv_1|^q+|\pi_1|^q\,dy\bigg)^{1/q}.
\end{aligned}
$$
Using this inequality together with \eqref{171101@D1} and \eqref{171101@D1b}, we obtain that
$$
\begin{aligned}
&\inf_{\substack{\theta\in \bR \\ \Theta\in \bR^d}}\bigg(\dashint_{B_{\kappa r}^+} |D_1v-\Theta |^q+|D_{y'}v|^q+|\pi-\theta |^q\,dy\bigg)^{1/q}\\
&\lesssim_{d,\lambda} \kappa\inf_{\Theta\in \bR^d} \bigg(\dashint_{B_r^+}|D_1 v-\Theta|^q+|D_{y'}v|^q\,dy\bigg)^{1/q} + \kappa^{-d/q}\dashint_{B_r^+}(|\cF_\alpha|+|\cG|)\,dy.
\end{aligned}
$$
Thus, from the definitions of $\cF_\alpha$ and $\cG$, and the fact that
$$
\dashint_{B_r^+}|b|\,dy=\dashint_{B_r^+}|b-b(0)|\,dy\le \varrho_0(r),
$$
we get
\begin{equation}		\label{171101@D4}
\begin{aligned}
&\inf_{\substack{\theta\in \bR \\ \Theta\in \bR^d}}\bigg(\dashint_{B_{\kappa r}^+} |D_1v-\Theta |^q+|D_{y'}v|^q+|\pi-\theta |^q\,dy\bigg)^{1/q}\\
&\lesssim \kappa\inf_{\Theta\in \bR^d} \bigg(\dashint_{B_r^+}|D_1 v-\Theta|^q+|D_{y'}v|^q\,dy\bigg)^{1/q}\\
&\quad + \kappa^{-d/q}\big(\|Dv\|_{L^\infty(B_r^+)}+\|\pi\|_{L^\infty(B_r^+)}\big)\bigg(\varrho_0(r)+\dashint_{B_r^+} \big|\cA^{\alpha\beta}-\overline{\cA^{\alpha\beta}}\big|\,dy\bigg)\\
&\quad + \kappa^{-d/q}\dashint_{B_r^+} \big( \big|F_\alpha-\overline{F_\alpha} \big|+\big|G-\overline{G}\big| \big)\,dy.
\end{aligned}
\end{equation}
We note that
$$
\sup_{y,z\in B_r^+}|D\Lambda(y)-D\Lambda(z)|\le \varrho_0(r), \quad \sup_{y\in B_r^+} |D\Lambda(y)|\le 1/2.
$$
Using this and following the proof of \cite[Lemma 2.1]{MR3747493}, we have
$$
\dashint_{B_r^+} \big|\cA^{\alpha\beta}-\overline{\cA^{\alpha\beta}}\big|\,dy\lesssim_{d,\lambda} \varrho_0(r)+\dashint_{B_r^+}\big|A^{\alpha\beta}(\Gamma)-\overline{A^{\alpha\beta}(\Gamma)}\big|\,dy.
$$
Hence, by the change of variables, \eqref{171101@D2a}, and $\varrho_0(r)\lesssim_{\varrho_0} \varrho_0(2r)$, we see that
$$
\dashint_{B_r^+} \big|\cA^{\alpha\beta}-\overline{\cA^{\alpha\beta}}\big|\,dy\lesssim_{d,\lambda,\varrho_0} \varrho_0(2r)+ \omega_{A^{\alpha\beta}}(2r).
$$
Similarly, we have
$$
\dashint_{B_r^+}\big(\big|F_\alpha-\overline{F_\alpha}\big|+\big|G-\overline{G}\big|\big)\,dy\lesssim_{d,\varrho_0} \|f_\alpha\|_{L^\infty(\Omega_{2r})}\varrho_0(2r)+\omega_{f_\alpha}(2r)+\omega_g(2r).
$$
Therefore, using the change of variables, \eqref{171127@eq3}, and \eqref{171101@D2a},
we get from \eqref{171101@D4} that
\begin{equation}		\label{171101@D5}
\begin{aligned}
&\inf_{\substack{\theta\in \bR \\ \Theta\in \bR^d}}\left(\dashint_{\Omega_{\kappa r/2}}|D_1 u-\Theta|^q+\sum_{i=2}^d|D_i \chi D_1u +D_i u|^q+|p-\theta|^q\,dx\right)^{1/q}\\
&\lesssim_{d,\lambda, R_0, \varrho_0} \kappa\inf_{\Theta\in \bR^d} \left(\dashint_{\Omega_{2r}}|D_1 u-\Theta|^q+\sum_{i=2}^d|D_i \chi D_1u +D_i u|^q\,dx\right)^{1/q}\\
&\quad +\kappa^{-d/q} \big(\|Du\|_{L^\infty(\Omega_{2r})}+\|p\|_{L^\infty(\Omega_{2r})}\big)(\varrho_0(2r)+\omega_{A^{\alpha\beta}}(2r))\\
&\quad +\kappa^{-d/q}(\|f_\alpha\|_{L^\infty(\Omega_{2r})}\varrho_0(2r)+\omega_{f_\alpha}(2r)+\omega_{g}(2r)).
\end{aligned}
\end{equation}
 for $0<r\le R_1/4$ and $\kappa\in (0,1/2]$.

{\em{Step 3}.}
We are ready to prove the lemma.
By replacing $\kappa/4$, $2r$, and $R_1/2$ by $\kappa$, $r$, and $2R_1$ in \eqref{171101@D5}, we obtain for $0<r\le 2R_1$ and $\kappa\in (0,1/8]$ that
$$
\begin{aligned}
\Psi(0,\kappa r)
&\le C_0 \kappa \Psi(0, r)\\
&\quad + C_0\kappa^{-d/q} \big(\|Du\|_{L^\infty(\Omega_{r})}+\|p\|_{L^\infty(\Omega_{r})}\big)(\varrho_0(r)+\omega_{A^{\alpha\beta}}(r))\\
&\quad +C_0\kappa^{-d/q}(\|f_\alpha\|_{L^\infty(\Omega_r)}\varrho_0(r)+\omega_{f_\alpha}(r)+\omega_{g}(r)),
\end{aligned}
$$
where $C_0=C_0(d,\lambda, R_0,\varrho_0)>0$.
We take $\kappa_2=\kappa_2(d,\lambda,\gamma,R_0,\varrho_0)\in (0,1/8]$ so that
$C_0\kappa_2^{1-\gamma}\le 1$.
Then for any $0<\kappa\le \kappa_2$,
we have
$$
\begin{aligned}
\Psi(0,\kappa r)
& \le \kappa^{\gamma} \Psi(0, r)+ C \big(\|Du\|_{L^\infty(\Omega_{r})}+\|p\|_{L^\infty(\Omega_{r})}\big)(\varrho_0(r)+\omega_{A^{\alpha\beta}}(r))\\
&\quad +C(\|f_\alpha\|_{L^\infty(\Omega_r)}\varrho_0(r)+\omega_{f_\alpha}(r)+\omega_{g}(r)),
\end{aligned}
$$
where $C=C(d,\lambda,\gamma, R_0,\varrho_0,\kappa)>0$.
By iterating, we obtain for $j\in \{1,2,\ldots\}$ that
\begin{equation}		\label{171229@eq3a}
\begin{aligned}
\Psi(0, \kappa^j r)
& \le \kappa^{\gamma j} \Psi(0, r) \\
&\quad + C \big(\|Du\|_{L^\infty(\Omega_{r})}+\|p\|_{L^\infty(\Omega_{r})}\big)(\tilde{\varrho}_0(\kappa^j r)+\tilde{\omega}_{A^{\alpha\beta}}(\kappa^j r))\\
&\quad +C(\|f_\alpha\|_{L^\infty(\Omega_r)} \tilde{\varrho}_0(\kappa^j r)+\tilde{\omega}_{f_\alpha}(\kappa^j r)+\tilde{\omega}_{g}(\kappa^j r)),
\end{aligned}
\end{equation}
where we used \eqref{171229@eq3} and
$$
\sum_{i=1}^j \kappa^{\gamma(i-1)}\varrho_0(\kappa^{j-i}r)\le \kappa^{-\gamma} \tilde{\varrho}_{0}(\kappa^j r).
$$
The estimate \eqref{171229@eq3a} corresponds to \eqref{171229@eq2}.
The rest of the proof is identical to that of Lemma \ref{171228@lem1} and is omitted.
\end{proof}

By combining Lemmas \ref{171228@lem1} and \ref{171101@lem1}, we obtain the following $L^q$-mean oscillation estimates for $Du$ and $p$.

\begin{lemma}		\label{171102@lem5}
Let $x_0\in \Omega$ and $\gamma\in (0,1)$.
Under the same hypothesis of Theorem \ref{M4} $(a)$, if $R_1=R_1(\varrho_0,R_0)$ is the constant from Lemma \ref{171101@lem1} and
$$
\kappa=\kappa(d,\lambda,\gamma,R_0, \varrho_0)=\min\{\kappa_1,\kappa_2\},
$$
where $\kappa_1$ and $\kappa_2$ are constants from Lemmas \ref{171228@lem1} and \ref{171101@lem1},
then  the following hold.

\begin{enumerate}[$(i)$]
\item
For any $0<r\le R_1$, we have
\begin{equation}		\label{171103@eq5}
\begin{aligned}
&\sum_{j=0}^\infty\Phi(x_0, \kappa^j r)\lesssim_{d,\lambda,\gamma, R_0, \varrho_0}  r^{-d} \big(\|Du\|_{L^1(\Omega_{3r}(x_0))}+\|p\|_{L^1(\Omega_{3r}(x_0))}\big) \\
&\quad +\big(\|Du\|_{L^\infty(\Omega_{3r}(x_0))}+\|p\|_{L^\infty(\Omega_{3r}(x_0))}\big) \int_0^r \frac{\varrho^\sharp_0(t) + \omega^\sharp_{A^{\alpha\beta}}(t)}{t}\,dt \\
&\quad + \|f_\alpha\|_{L^\infty(\Omega_{3r}(x_0))}\int_0^r \frac{\varrho_0^\sharp(t)}{t}\,dt+\int_0^r \frac{\omega_{f_\alpha}^\sharp(t)+{\omega}_{g}^\sharp(t)}{t}\,dt,
\end{aligned}
\end{equation}
where each integration is finite; see Remark \ref{180106@rmk1}
\item
For any $0<\rho\le r\le R_1$,  we have
\begin{equation}		\label{171102@eq2a}
\begin{aligned}
&\Phi(x_0, \rho)\lesssim_{d,\lambda,\gamma,R_0, \varrho_0} \left(\frac{\rho}{r}\right)^\gamma r^{-d} \big(\|Du\|_{L^1(\Omega_{3r}(x_0))}+\|p\|_{L^1(\Omega_{3r}(x_0))}\big)\\
&\quad +\big(\|Du\|_{L^\infty(\Omega_{3r}(x_0))}+\|p\|_{L^\infty(\Omega_{3r}(x_0))}\big)(\varrho^\sharp_0(\rho)+{\omega}^\sharp_{A^{\alpha\beta}}(\rho))\\
&\quad + \|f_\alpha\|_{L^\infty(\Omega_{3r}(x_0))}{\varrho}_0^\sharp(\rho)+{\omega}_{f_\alpha}^\sharp(\rho)+{\omega}_{g}^\sharp(\rho).
\end{aligned}
\end{equation}
\end{enumerate}
Here, we set
$$
\varrho^\sharp_0(\rho):=\sup_{\rho\le R\le R_1}\left(\frac{\rho}{R}\right)^\gamma\tilde{\varrho}_0(R), \quad
\omega^\sharp_{\bullet}(\rho):=\sup_{\rho\le R\le R_1}\left(\frac{\rho}{R}\right)^\gamma \tilde{\omega}_{\bullet}(R).
$$
\end{lemma}

\begin{remark}		\label{180106@rmk1}
Note that $\varrho^\sharp_0$ is a Dini function; see \cite[pp. 463--464]{MR3747493}.
By the definition of $\varrho_0^\sharp$ and \eqref{171102@eq3}, we have
$$
2^{-\gamma}\varrho_0^\sharp(t)\le  \varrho_0^\sharp(s)\lesssim_{\gamma,\varrho_0} \varrho_0^\sharp(t), \quad \frac{t}{2}\le s\le t \le R_1.
$$
Therefore, using the comparison principle for Riemann integrals, we get
\begin{equation}		\label{171103@eq1}
\sum_{j=0}^\infty \varrho_0^\sharp(\kappa^j r)\lesssim_{\gamma, \varrho_0,\kappa}\int_0^r \frac{\varrho_0^\sharp(t)}{t}\,dt<\infty, \quad 0<r \le R_1.
\end{equation}
Similarly, we have
\begin{equation}		\label{171103@eq1a}
\sum_{j=0}^\infty \omega_{f}^\sharp(\kappa^j r)\lesssim_{d,\gamma,R_0, \varrho_0,\kappa}\int_0^r \frac{\omega_{f}^\sharp(t)}{t}\,dt<\infty, \quad 0<r\le R_1,
\end{equation}
for any $f$ having Dini mean oscillation in $\Omega$.
\end{remark}

\begin{proof}[Proof of Lemma \ref{171102@lem5}]
The estimate \eqref{171103@eq5} is an easy consequence of the estimate \eqref{171102@eq2a}.
Indeed, for $j\in \{0,1,2,\ldots\}$, by taking $\rho=\kappa^j r$ in \eqref{171102@eq2a}, we have
\begin{equation}		\label{171103@eq5a}
\begin{aligned}
&\Phi(x_0, \kappa^j r)\lesssim \kappa^{\gamma j} r^{-d} \big(\|Du\|_{L^1(\Omega_{3r}(x_0))}+\|p\|_{L^1(\Omega_{3r}(x_0))}\big)\\
&\quad +\big(\|Du\|_{L^\infty(\Omega_{3r}(x_0))}+\|p\|_{L^\infty(\Omega_{3r}(x_0))}\big)(\varrho^\sharp_0(\kappa^j r)+{\omega}^\sharp_{A^{\alpha\beta}}(\kappa^j r))\\
&\quad + \|f_\alpha\|_{L^\infty(\Omega_{3r}(x_0))}{\varrho}_0^\sharp(\kappa^j r)+{\omega}_{f_\alpha}^\sharp(\kappa^j r)+{\omega}_{g}^\sharp(\kappa^j r).
\end{aligned}
\end{equation}
Taking the summations of both sides of \eqref{171103@eq5a} with respect to $j=0,1,2,\ldots$, and using  \eqref{171103@eq1} and \eqref{171103@eq1a}, we conclude \eqref{171103@eq5}.

To complete the proof,
it suffices to prove that \eqref{171102@eq2a} holds.
Without loss of generality, we assume that $x_0=0\in {\Omega}$.
We denote $B_R=B_R(0)$ and $\Omega_R=\Omega_R(0)$.
Let $0<\rho\le r \le R_1$.
Note that if $r/6< \rho\le r$, then \eqref{171102@eq2a} follows from the definition of $\Phi$.
Hence we only need to consider the case of $0<\rho\le r/6$.
We consider the following three cases:
$$
r\le \operatorname{dist}(0,\partial \Omega), \quad \operatorname{dist}(0,\partial \Omega)\le 4\rho , \quad 4\rho  <  \operatorname{dist}(0,\partial \Omega) <  r.
$$
\begin{enumerate}[i.]
\item
$r\le \operatorname{dist}(0,\partial \Omega)$:
Set $R=r/4$.
Since $B_{4R}\subset \Omega$,
by Lemma \ref{171228@lem1} $(ii)$, we have
$$
\Phi(0,\rho)\lesssim \left(\frac{\rho}{R}\right)^{\gamma} \Phi(0, R) +\|Du\|_{L^\infty(B_R)}\tilde{\omega}_{A^{\alpha\beta}}(\rho)+\tilde{\omega}_{f_\alpha}(\rho)+\tilde{\omega}_{g}(\rho).
$$
Thus from the fact that
$$
\tilde{\omega}_{\bullet}(\rho)\le \omega^\sharp_{\bullet}(\rho), \quad \Phi(0, R)\lesssim R^{-d}\big(\|Du\|_{L^1(\Omega_R)}+\|p\|_{L^1(\Omega_R)}\big),
$$
we get \eqref{171102@eq2a}.

\item
$\operatorname{dist}(0,\partial \Omega)\le 4\rho$:
We take $y_0\in \partial \Omega$ such that $\operatorname{dist}(0,\partial \Omega)=|y_0|$.
We fix a $C^{1,\rm{Dini}}$ function $\chi$ and a coordinate system associated with $y_0$ satisfying \eqref{171101@E1} and \eqref{171101@E2}.
In this coordinate system, using \eqref{171127@eq3} and the fact that $\Omega_\rho \subset \Omega_{5\rho}(y_0)$, we have
$$
\Phi(0, \rho) \lesssim_{d,R_0, \varrho_0} \Psi(y_0,5\rho)+\bigg(\dashint_{\Omega_{5\rho}(y_0)} \sum_{i=2}^d|D_i \chi D_1 u|^q\,dx\bigg)^{1/q},
$$
where $\Psi$ is given in Lemma \ref{171101@lem1}.
Note that
$$
|D_{x'} \chi(x')|=|D_{x'}\chi(x')-D_{x'} \chi(y_0')|\le \varrho_0(5\rho), \quad x'\in B_{5\rho}'(y_0').
$$
Using this together with Lemma \ref{171101@lem1} $(ii)$, we obtain that
\begin{align}		
\label{171129@eq2}
\Phi(0,\rho)&\lesssim \Psi(y_0, 5\rho)+ \varrho_0(5\rho) \|Du\|_{L^\infty(\Omega_{5\rho}(y_0))}\\
\nonumber
&\lesssim \left(\frac{\rho}{r}\right)^{\gamma}\Psi(y_0,r)+\big(\||Du|+|p|\|_{L^\infty(\Omega_r(y_0))}\big) \big(\tilde{\varrho}_0(5\rho)+\tilde{\omega}_{A^{\alpha\beta}}(5\rho)\big)\\
\label{171129@eq1}
& \quad + \|f_\alpha\|_{L^\infty(\Omega_r(y_0))}\tilde{\varrho}_0(5\rho)+\tilde{\omega}_{f_\alpha}(5\rho)+\tilde{\omega}_g(5\rho).
\end{align}

Since it holds that
$$
\Omega_r(y_0)\subset \Omega_{3r}, \quad \tilde{\varrho}(5\rho)\lesssim_{\gamma} \varrho_0^\sharp(\rho), \quad \tilde{\omega}_{\bullet}(5\rho)\lesssim_{\gamma} \omega_{\bullet}^\sharp(\rho),
$$
$$
\Psi(y_0, r)\lesssim r^{-d}\big(\|Du\|_{L^1(\Omega_{3r})}+\|p\|_{L^1(\Omega_{3r})}\big),
$$
we get \eqref{171102@eq2a} from \eqref{171129@eq1}.

\item
$4\rho < \operatorname{dist}(0, \partial \Omega) < r$:
Set $R=\operatorname{dist}(0, \partial \Omega)/4$, and observe that
$$
\rho < R, \quad 5R < 2r\le 2R_1.
$$
Since $B_{4R}\subset \Omega$, by Lemma \ref{171228@lem1} $(ii)$, we have
\begin{equation}		\label{171129@eq1b}
\Phi(0,\rho)\lesssim \left(\frac{\rho}{R}\right)^{\gamma} \Phi(0, R) + \|Du\|_{L^\infty(B_R)} \tilde{\omega}_{A^{\alpha\beta}}(\rho)+\tilde{\omega}_{f_\alpha}(\rho)+\tilde{\omega}_{g}(\rho).
\end{equation}
We take $y_0\in \partial \Omega$ such that $\operatorname{dist}(0,\partial \Omega)=|y_0|$.
We fix a $C^{1,\rm{Dini}}$ function $\chi$ and a coordinate system associated with $y_0$ satisfying \eqref{171101@E1} and \eqref{171101@E2}.
In this coordinate system, similar to \eqref{171129@eq1}, we have
\begin{align}
\nonumber
\Phi(0, R)
&\lesssim \Psi(y_0, 5R)+\varrho_0(5R)\|Du\|_{L^\infty(\Omega_{5R}(y_0))}\\
\nonumber
&\lesssim \left(\frac{R}{r}\right)^\gamma \Psi(y_0, 2r)+\big(\||Du|+|p|\|_{L^\infty(\Omega_{2r}(y_0))}\big)\big(\tilde{\varrho}_0(5R)+\tilde{\omega}_{A^{\alpha\beta}}(5R)\big)\\
\label{171129@eq1a}
&\quad + \|f_\alpha\|_{L^\infty(\Omega_{2r}(y_0))}\tilde{\varrho}_0(5R)+\tilde{\omega}_{f_\alpha}(5R)+\tilde{\omega}_g(5R).
\end{align}
Combining \eqref{171129@eq1b} and \eqref{171129@eq1a},
and using the fact that
$$
\Omega_{2r}(y_0)\subset \Omega_{3r}, \quad \Psi(y_0, 2r)\lesssim r^{-d}\big(\|Du\|_{L^1(\Omega_{3r})}+\|p\|_{L^1(\Omega_{3r})}\big),
$$
we get \eqref{171102@eq2a}.
\end{enumerate}
The lemma is proved.
\end{proof}

Now we are ready to prove the assertion $(a)$ in the theorem.

\begin{proof}[Proof of Theorem \ref{M4} $(a)$]
In this proof, we fix $\gamma\in (0, 1)$.
Let $R_1=R_1(\varrho_0, R_0)\in (0, R_0/4)$ be the constant from Lemma \ref{171101@lem1} and $\kappa=\kappa(d,\lambda, \gamma,R_0, \varrho_0)\in (0,1/8]$  be the constant from Lemma \ref{171102@lem5}.
We denote
$$
\cU=|Du|+|p|, \quad \cG(r)=\int_0^r \frac{\omega^\sharp_{f_\alpha}(t)+\omega^\sharp_{g}(t)}{t}\,dt.
$$

We first derive $L^\infty$-estimates for $Du$ and $p$.
Let $x_0\in {\Omega}$ and $0<r \le R_1$.
We take $\theta_{x_0, r}\in \bR$ and $\Theta_{x_0,r}\in \bR^{d\times d}$ to be such that
$$
\Phi(x_0, r)=\bigg(\dashint_{\Omega_r(x_0)} |Du-\Theta_{x_0, r}|^q+|p-\theta_{x_0,r}|^q\, dx\bigg)^{1/q}.
$$
Similarly, we find $\theta_{x_0,\kappa^ir}\in \bR$ and $\Theta_{x_0, \kappa^i r}\in \bR^{d\times d}$ for $i\in \{1,2,\ldots\}$.
Recall the assumption that $(u,p)\in C^1(\overline{\Omega})^d\times C(\overline{\Omega})$.
Thus, since the right-hand side of \eqref{171103@eq5a} goes to zero as $j\to \infty$, we see that
\begin{equation}		\label{171103@eq7d}
\lim_{i\to \infty} \theta_{x_0,\kappa^i r}=p(x_0), \quad \lim_{i\to \infty}\Theta_{x_0, \kappa^i r}=Du(x_0).
\end{equation}
By averaging the inequality
$$
|\Theta_{x_0, \kappa r}-\Theta_{x_0,r}|^q\le |Du-\Theta_{x_0,\kappa r}|^q+|Du-\Theta_{x_0,r}|^q
$$
on $\Omega_{\kappa r}(x_0)$ and taking the $q$-th root, we have
$$
|\Theta_{x_0,\kappa r}-\Theta_{x_0,r}| \lesssim \Phi(x_0, \kappa r)+\Phi(x_0,r).
$$
Similarly, we have $
|\theta_{x_0,\kappa r}-\theta_{x_0,r}| \lesssim \Phi(x_0, \kappa r)+\Phi(x_0,r)$.
Thus by iterating and \eqref{171103@eq7d}, we have
\begin{equation}		\label{171103@eq7e}
|Du(x_0)-\Theta_{x_0,r}|+|p(x_0)-\theta_{x_0,r}| \lesssim \sum_{j=0}^\infty \Phi(x_0,\kappa^j r).
\end{equation}
This inequality together with Lemma \ref{171102@lem5} $(i)$ implies
$$
\begin{aligned}
&|Du(x_0)-\Theta_{x_0,r}|+|p(x_0)-\theta_{x_0,r}|  \\
&\lesssim  r^{-d} \|\cU\|_{L^1(\Omega_{3r}(x_0))}+\|\cU\|_{L^\infty(\Omega_{3r}(x_0))} \int_0^r \frac{\varrho^\sharp_0(t) + \omega^\sharp_{A^{\alpha\beta}}(t)}{t}\,dt \\
&\quad + \|f_\alpha\|_{L^\infty(\Omega_{3r}(x_0))}\int_0^r \frac{\varrho_0^\sharp(t)}{t}\,dt+\cG(r).
\end{aligned}
$$
Note that
$$
|\Theta_{x_0,r}|+|\theta_{x_0,r}|\lesssim \Phi(x_0,r)+r^{-d}\|\cU\|_{L^1(\Omega_r(x_0))} \lesssim r^{-d}\|\cU\|_{L^1(\Omega_r(x_0))}.
$$
Combining the above two inequalities, we have
$$
\begin{aligned}
\cU(x_0)  &\le C_1 r^{-d} \|\cU\|_{L^1(\Omega_{3r}(x_0))}+ C_1 \|\cU\|_{L^\infty(\Omega_{3r}(x_0))} \int_0^r \frac{\varrho^\sharp_0(t) + \omega^\sharp_{A^{\alpha\beta}}(t)}{t}\,dt \\
&\quad + C_1 \|f_\alpha\|_{L^\infty(\Omega_{3r}(x_0))}\int_0^r \frac{\varrho_0^\sharp(t)}{t}\,dt+ C_1\cG(r),
\end{aligned}
$$
where $C_1=C_1(d,\lambda,\gamma, R_0, \varrho_0)$.
We take $r_0\in (0, R_1]$  so that
$$
C_1\int_0^{r_0}  \frac{\varrho^\sharp_0(t) + \omega^\sharp_{A^{\alpha\beta}}(t)}{t}\,dt\le \frac{1}{3^d}.
$$
Then for any $x_0 \in \Omega$ and $0 < r \le r_0$, we have that
\begin{equation}		\label{171127@eq2}
\begin{aligned}
\cU(x_0)  &\le C_1 r^{-d} \|\cU\|_{L^1(\Omega_{3r}(x_0))}+ 3^{-d} \|\cU\|_{L^\infty(\Omega_{3r}(x_0))}  \\
&\quad + 3^{-d} \|f_\alpha\|_{L^\infty(\Omega_{3r}(x_0))}+ C_1\cG(r).
\end{aligned}
\end{equation}
Here, the constant $r_0$ depends only on  $d$, $\lambda$, $\gamma$, $R_0$, $\varrho_0$, and $\omega_{A^{\alpha\beta}}$.

Now let us fix $x_0\in \Omega$ and $0<R\le R_1$.
For $k\in \{2,3,\ldots\}$, we denote $r_k=R(1-2^{1-k})$.
Since $r_{k+1}-r_k=2^{-k}R$, we have $\Omega_{4r}(y)\subset \Omega_{r_{k+1}}(x_0)$ for any $y\in \Omega_{r_k}(x_0)$ and $r= 2^{-k-2} R$.
We take $k_0$ sufficiently large such that $ 2^{-k_0-2} R_1\le r_0$.
Then by \eqref{171127@eq2} with $r=2^{-k-2} R$, we have for $k\ge k_0$ that
$$
\begin{aligned}
\|\cU\|_{L^\infty(\Omega_{r_{k}}(x_0))} &\le C_1 \left(\frac{2^{k+2}}{R}\right)^{d} \|\cU\|_{L^1(\Omega_{r_{k+1}}(x_0))}+3^{-d} \|\cU\|_{L^\infty(\Omega_{r_{k+1}}(x_0))}\\
&\quad +3^{-d} \|f_\alpha\|_{L^\infty(\Omega_{r_{k+1}}(x_0))}+C_1\cG(R).
\end{aligned}
$$
By multiplying both sides of the above inequality by $3^{-dk}$ and summing the terms with respect to $k=k_0,k_0+1,\ldots$, we see that
$$
\begin{aligned}
\sum_{k=k_0}^\infty 3^{-dk}\|\cU\|_{L^\infty(\Omega_{r_{k}}(x_0))}
&\le  C R^{-d} \|\cU\|_{L^1(\Omega_{R}(x_0))}
+\sum_{k=k_0+1}^\infty 3^{-dk} \|\cU\|_{L^\infty(\Omega_{r_{k}}(x_0))}\\
&\quad +C \|f_\alpha\|_{L^\infty(\Omega_{R}(x_0))}+C\cG(R),
\end{aligned}
$$
where each summation is finite and $C=C(d,\lambda, \gamma,R_0, \varrho_0)>0$.
By subtracting
$$
\sum_{k=k_0+1}^\infty 3^{-dk}\|\cU\|_{L^\infty(\Omega_{r_k}(x_0))}
$$
from both sides of the above inequality, we get the following $L^\infty$-estimate for $Du$ and $p$:
\begin{equation}		\label{171103@eq7}
\|\cU\|_{L^\infty(\Omega_{R/2}(x_0))} \le C \big(R^{-d} \|\cU\|_{L^1(\Omega_{R}(x_0))} +\|f_\alpha\|_{L^\infty(\Omega_{R}(x_0))}+\cG(R)\big)
\end{equation}
for any $x_0\in {\Omega}$ and $R\in (0,R_1]$,
where $C=C(d,\lambda,\gamma, R_0, \varrho_0, \omega_{A^{\alpha\beta}})$.

Next, we shall derive estimates of the modulus of continuity of $Du$ and $p$.
We first claim that for any $x\in \Omega$ and $0<\rho\le r\le R_1/4$, we have
\begin{equation}		\label{171103@eq7a}
\begin{aligned}
&\sum_{j=0}^\infty \Phi(x,\kappa^j \rho) \lesssim_{d,\lambda,\gamma,R_0, \varrho_0} \left(\frac{\rho}{r}\right)^{\gamma} r^{-d}\|\cU\|_{L^1(\Omega_{10r}(x))}\\
&\quad +\big(\|\cU\|_{L^\infty(\Omega_{10r}(x))}+\|f_\alpha\|_{L^\infty(\Omega_{10r}(x))}\big)\int_0^\rho \frac{\varrho^\sharp_0(t)+\omega^\sharp_{A^{\alpha\beta}}(t)}{t}\,dt+\cG(\rho).
\end{aligned}
\end{equation}
We consider the following two cases:
$$
4\rho\le \operatorname{dist}(x,\partial \Omega) \quad \text{and} \quad 4\rho>\operatorname{dist}(x,\partial \Omega).
$$
\begin{enumerate}[i.]
\item
$4\rho\le \operatorname{dist}(x,\partial \Omega)$:
Since $B_{4\rho}(x) \subset \Omega$,
by Lemma \ref{171228@lem1} $(i)$, we have
$$
\begin{aligned}
\sum_{j=0}^\infty \Phi(x, \kappa^j \rho)&\lesssim \Phi(x,\rho)+\|Du\|_{L^\infty(B_\rho(x))}\int_0^\rho \frac{\tilde{\omega}_{A^{\alpha\beta}}(t)}{t}\,dt\\
&\quad +\int_0^\rho\frac{\tilde{\omega}_{f_\alpha}(t)+\tilde{\omega}_g(t)}{t}\,dt.
\end{aligned}
$$
From Lemma \ref{171102@lem5} $(ii)$, it follows that
$$
\begin{aligned}
\Phi(x,\rho)&\lesssim \left(\frac{\rho}{r}\right)^{\gamma}r^{-d} \|\cU\|_{L^1(\Omega_{3r}(x))}+\|\cU\|_{L^\infty(\Omega_{3r}(x))}(\varrho^\sharp_0(\rho)+\omega^\sharp_{A^{\alpha\beta}}(\rho))\\
&\quad +\|f_\alpha\|_{L^\infty(\Omega_{3r}(x))}\varrho^\sharp_0(\rho)+\omega^\sharp_{f_\alpha}(\rho)+\omega_g^\sharp(\rho).
\end{aligned}
$$
Combining the above two inequalities, and using the fact that
\begin{equation}		\label{180314@A1}
\tilde{\omega}_{\bullet}(\rho)\le \omega_\bullet^\sharp(\rho) \lesssim \int_0^\rho \frac{\omega_\bullet^\sharp(t)}{t}\,dt,
\quad \varrho_0^\sharp(\rho) \lesssim \int_0^\rho \frac{\varrho_0^\sharp(t)}{t}\,dt,
\end{equation}
we get
\begin{equation}		\label{171124@eq1}
\begin{aligned}
&\sum_{j=0}^\infty \Phi(x, \kappa^j \rho)\lesssim \left(\frac{\rho}{r}\right)^{\gamma} r^{-d} \|\cU\|_{L^1(\Omega_{3r}(x))}\\
&\quad +\big(\|\cU\|_{L^\infty(\Omega_{3r}(x))}+\|f_\alpha\|_{L^\infty(\Omega_{3r}(x))}\big) \int_0^\rho \frac{\varrho^\sharp_0(t)+\omega^\sharp_{A^{\alpha\beta}}(t)}{t}\,dt+\cG(\rho).
\end{aligned}
\end{equation}
This inequality implies \eqref{171103@eq7a}.

\item
$4\rho>\operatorname{dist}(x, \partial \Omega)$:
Let $i_0$ be the integer such that $4\kappa^{i_0+1}\rho\le \operatorname{dist}(x, \partial \Omega)<4\kappa^{i_0}\rho$.
Since $B_{4\kappa^{i_0+1}\rho}(x)\subset \Omega$, by the same reasoning as in \eqref{171124@eq1}, we have
$$
\begin{aligned}
&\sum_{j=i_0+1}^\infty \Phi(x,\kappa^j \rho)=\sum_{j=0}^\infty \Phi(x, \kappa^{j+i_0+1} \rho) \lesssim \left(\frac{\kappa^{i_0+1}\rho}{r}\right)^{\gamma} r^{-d} \|\cU\|_{L^1(\Omega_{3r}(x))} \\
&\quad +\big(\|\cU\|_{L^\infty(\Omega_{3r}(x))}+\|f_\alpha\|_{L^\infty(\Omega_{3r}(x))}\big)
\int_0^{\kappa^{i_0+1}\rho} \frac{\varrho^\sharp_0(t)+\omega^\sharp_{A^{\alpha\beta}}(t)}{t}\,dt+\cG(\kappa^{i_0+1}\rho).
\end{aligned}
$$
Thus we get (using $\kappa^{i_0+1}\rho\le \rho$)
\begin{equation}		\label{171124@eq1a}
\begin{aligned}
&\sum_{j=i_0+1}^\infty \Phi(x,\kappa^j \rho) \lesssim \left(\frac{\rho}{r}\right)^{\gamma} r^{-d} \|\cU\|_{L^1(\Omega_{3r}(x))} \\
&\quad +\big(\|\cU\|_{L^\infty(\Omega_{3r}(x))}+\|f_\alpha\|_{L^\infty(\Omega_{3r}(x))}\big)
\int_0^{\rho} \frac{\varrho^\sharp_0(t)+\omega^\sharp_{A^{\alpha\beta}}(t)}{t}\,dt+\cG(\rho).
\end{aligned}
\end{equation}
We take $y_0\in \partial \Omega$ such that $|y_0|=\operatorname{dist}(x, \partial \Omega)$.
We fix a coordinate system associated with $y_0$ satisfying \eqref{171101@E2}.
Observe that for $j\in \{0,1,\ldots,i_0\}$, we have
$$
\Omega_{\kappa^j \rho}(x) \subset \Omega_{5\kappa^j \rho}(y_0).
$$
Then similar to \eqref{171129@eq2}, we obtain
$$
\Phi(x, \kappa^j \rho)\lesssim \Psi(y_0, 5\kappa^j \rho)+\varrho_0(5\kappa^j \rho)\|Du\|_{L^\infty(\Omega_{5\rho}(y_0))}.
$$
Summing the terms with respect to $j=0,1,\ldots,i_0$, and using the fact that
$$
\sum_{j=0}^{i_0} \varrho_0(5\kappa^j \rho)\le \sum_{j=0}^\infty \tilde{\varrho}_0(5\kappa^j \rho)\lesssim \int_0^{5\rho}\frac{\varrho^\sharp_0(t)}{t}\,dt ,
$$
we have
\begin{equation}		\label{180314@A2}
\sum_{j=0}^{i_0} \Phi(x,\kappa^j \rho)\lesssim \sum_{j=0}^{i_0} \Psi(y_0, 5\kappa^j \rho)+\|Du\|_{L^\infty(\Omega_{5\rho}(y_0))}\int_0^{5\rho}\frac{\varrho^\sharp_0(t)}{t}\,dt.
\end{equation}
Recall that $0<5\rho\le 5r\le 2R_1$.
Hence, by Lemma \ref{171101@lem1} and \eqref{180314@A1}, we get the following two inequalities:
$$
\begin{aligned}
&\sum_{j=0}^{i_0} \Psi(x,5\kappa^j \rho) \lesssim \Psi(y_0, 5\rho)\\
&\quad +\big(\|\cU\|_{L^\infty(\Omega_{5\rho}(y_0))}+\|f_\alpha\|_{L^\infty(\Omega_{5\rho}(y_0))}\big)
\int_0^{5\rho} \frac{\varrho^\sharp_0(t)+\omega^\sharp_{A^{\alpha\beta}}(t)}{t}\,dt+\cG(5\rho),
\end{aligned}
$$
$$
\begin{aligned}
&\Psi(y_0, 5\rho)\lesssim \left(\frac{\rho}{r}\right)^\gamma\Psi(y_0, 5r)\\
&\quad +\big(\|\cU\|_{L^\infty(\Omega_{5r}(y_0))}+\|f_\alpha\|_{L^\infty(\Omega_{5r}(y_0))}\big)
\int_0^{5\rho} \frac{\varrho^\sharp_0(t)+\omega^\sharp_{A^{\alpha\beta}}(t)}{t}\,dt+\cG(5\rho).
\end{aligned}
$$
Combining these together, we get from \eqref{180314@A2} that
\begin{equation}		\label{171104@eq1}
\begin{aligned}
&\sum_{j=0}^{i_0} \Phi(x,\kappa^j \rho) \lesssim \left(\frac{\rho}{r}\right)^{\gamma} r^{-d} \|\cU\|_{L^1(\Omega_{10r}(x))}\\
&\quad+\big(\|\cU\|_{L^\infty(\Omega_{10r}(x))}+\|f_\alpha\|_{L^\infty(\Omega_{10r}(x))}\big)
\int_0^{\rho} \frac{{\varrho}^\sharp_0(t)+ {\omega}^\sharp_{A^{\alpha\beta}}(t)}{t}\,dt+\cG(\rho),
\end{aligned}
\end{equation}
where we used the fact that $\Omega_{5r}(y_0)\subset \Omega_{10r}(x)$,
$$
\int_0^{5\rho}\frac{\varrho^\sharp_0(t)}{t}\,dt \lesssim  \int_0^\rho \frac{\varrho_0^\sharp(t)}{t}\,dt, \quad
\int_0^{5\rho}\frac{\omega^\sharp_\bullet (t)}{t}\,dt \lesssim  \int_0^\rho \frac{\omega_{\bullet}^\sharp(t)}{t}\,dt.
$$
Therefore, we get \eqref{171103@eq7a} from \eqref{171124@eq1a} and \eqref{171104@eq1}.
\end{enumerate}

Now we are ready to estimate the modulus of continuity of $Du$ and $p$.
Let $x_0\in {\Omega}$ and $0<R\le R_1$.
Let $x,y\in \Omega_{R/4}(x_0)$ with $\rho:=|x-y|\le R/40$.
Then for any $z\in \Omega_\rho(x)\cap \Omega_\rho(y)$, we have
\begin{align*}
&|Du(x)-Du(y)|^q \\
&\le |Du(x)-\Theta_{x,\rho}|^q+|\Theta_{x,\rho}-\Theta_{y,\rho}|^q + |Du(y)-\Theta_{y,\rho}|^q\\
&\le 2\sup_{y_0\in \Omega_{R/4}(x_0)}|Du(y_0)-\Theta_{y_0,\rho}|^q+|Du(z)-\Theta_{x,\rho}|^q+|Du(z)-\Theta_{y,\rho}|^q.
\end{align*}
By taking average over $z\in \Omega_\rho(x)\cap \Omega_\rho(y)$ and taking the $q$-th root, we have
\begin{align*}
|Du(x)-Du(y)|& \lesssim \sup_{y_0\in \Omega_{R/4}(x_0)}|Du(y_0)-\Theta_{y_0,\rho}|+\Phi(x,\rho)+\Phi(y,\rho)\\
&\lesssim \sup_{y_0\in \Omega_{R/4}(x_0)} \Bigg(\sum_{j=0}^\infty \Phi(y_0, \kappa^j \rho)+\Phi(y_0, \rho)\Bigg)\\
&\lesssim \sup_{y_0\in \Omega_{R/4}(x_0)} \sum_{j=0}^\infty \Phi(y_0, \kappa^j \rho)
\end{align*}
where we used \eqref{171103@eq7e} in the second inequality.
Similarly, we get the same bound for $p$, and thus, by using \eqref{171103@eq7a} and the fact that
$$
\Omega_{R/4}(y_0)\subset \Omega_{R/2}(x_0) \quad \text{for }\, y_0\in \Omega_{R/4}(x_0),
$$
we obtain
$$
\begin{aligned}
&|Du(x)-Du(y)|+|p(x)-p(y)| \lesssim \left(\frac{\rho}{R}\right)^{\gamma} R^{-d}\|\cU\|_{L^1(\Omega_{R/2}(x_0))}\\
&\quad +\big(\|\cU\|_{L^\infty(\Omega_{R/2}(x_0))}+\|f_\alpha\|_{L^\infty(\Omega_{R/2}(x_0))}\big)
\int_0^\rho \frac{\varrho^\sharp_0(t)+\omega^\sharp_{A^{\alpha\beta}}(t)}{t}\,dt+\cG(\rho).
\end{aligned}
$$
Therefore, by \eqref{171103@eq7}, we have
\begin{equation}		\label{171128@C1}
\begin{aligned}
&|Du(x)-Du(y)|+|p(x)-p(y)|\\
&\le C R^{-d}\|\cU\|_{L^1(\Omega_{R}(x_0))}\left(\left(\frac{|x-y|}{R}\right)^{\gamma} +\int_0^{|x-y|} \frac{\varrho^\sharp_0(t)+\omega^\sharp_{A^{\alpha\beta}}(t)}{t}\,dt\right) \\
&\quad +C\|f_\alpha\|_{L^\infty(\Omega_{R}(x_0))}\int_0^{|x-y|} \frac{\varrho^\sharp_0(t)+\omega^\sharp_{A^{\alpha\beta}}(t)}{t}\,dt\\
&\quad + C\cG(R) \int_0^{|x-y|} \frac{\varrho^\sharp_0(t)+\omega^\sharp_{A^{\alpha\beta}}(t)}{t}\,dt+C\cG(|x-y|)
\end{aligned}
\end{equation}
for any $x,y\in \Omega_{R/4}(x_0)$ with $|x-y|\le R/40$, where $x_0\in {\Omega}$, $0<R\le R_1$, and $C>0$ is a constant depending only on $d$, $\lambda$, $\gamma$, $R_0$, $\varrho_0$, and $\omega_{A^{\alpha\beta}}$.
We note that if $x,y\in \Omega_{R/4}(x_0)$ with $|x-y|>R/40$, then by \eqref{171103@eq7}, we have
\begin{equation}		\label{180315@eq1}
\begin{aligned}	
&|Du(x)-Du(y)|+|p(x)-p(y)| \\
&\le C \left(\frac{|x-y|}{R}\right)^\gamma \big(R^{-d}\|\cU \|_{L^1(\Omega_{R}(x_0))}+\|f_\alpha \|_{L^\infty(\Omega_{R}(x_0))}+\cG(R)\big).
\end{aligned}
\end{equation}
The assertion $(a)$ in Theorem \ref{M4} is proved.
\end{proof}

We now turn to the proof of the assertion $(b)$ in the theorem.

\begin{proof}[Proof of Theorem \ref{M4} $(b)$]
In this proof, we set $\gamma=\frac{1+\gamma_0}{2}$ and $\varrho_0(r) = Nr^{\gamma_0}$, where $\gamma_0\in (0,1)$ and $N>0$.
Let $R_1=R_1(\varrho_0,R_0)\in (0, R_0/4)$ be the constant from Lemma \ref{171101@lem1} and
$\kappa=\kappa(d,\lambda,\gamma,R_0, \varrho_0)\in (0,1/8]$
be the constant from Lemma \ref{171102@lem5}.
Here, we note that
$$
R_1=R_1(\gamma_0, N, R_0) \quad \text{and}\quad \kappa=\kappa(d,\lambda, \gamma_0, N, R_0).
$$
By the same reasoning as in \cite[Lemma 8.1 $(b)$]{arXiv:1803.05560}, we have
$$
\tilde{\varrho}_0(r)=\varrho_0(r)+\sum_{i=1}^\infty \kappa^{\gamma i}\big(\varrho_{0}(\kappa^{-i}r)[\kappa^{-i} r<1]+\varrho_{0}(1)[\kappa^{-i}r\ge 1]\big)\lesssim_{\kappa,\gamma_0,N} r^{\gamma_0}
$$
and
$$
\tilde{\omega}_f(r)=\sum_{i=1}^\infty \kappa^{\gamma i}\big(\omega_{f}(\kappa^{-i}r)[\kappa^{-i} r<1]+\omega_{f}(1)[\kappa^{-i}r\ge 1]\big)\lesssim_{\kappa,\gamma_0}  [f]_{C^{\gamma_0}(\Omega)} r^{\gamma_0}
$$
for any function $f$ satisfying $[f]_{C^{\gamma_0}(\Omega)}<\infty$ and $0<r\le R_1$.
Then it follows from the definitions of $\varrho_0^\sharp$ and $\omega^\sharp_f$ that
$$
{\varrho}_0^\sharp (r)\lesssim r^{\gamma_0} , \quad {\omega}^\sharp_f(r)\lesssim  [f]_{C^{\gamma_0}(\Omega)} r^{\gamma_0}.
$$
Therefore, by \eqref{171103@eq7}, \eqref{171128@C1}, and \eqref{180315@eq1}, we conclude that
$$
\begin{aligned}
&\|Du\|_{L^\infty(\Omega_{R/2}(x_0))}+\|p\|_{L^\infty(\Omega_{R/2}(x_0))} +R^{\gamma_0}\big([Du]_{C^{\gamma_0}(\Omega_{R/4}(x_0))}+[p]_{C^{\gamma_0}(\Omega_{R/4}(x_0))}\big)\\
&\le C R^{-d} \big(\|Du\|_{L^1(\Omega_{R}(x_0))}+\|p\|_{L^1(\Omega_R(x_0))}\big)\\
&\quad +C\|f_\alpha\|_{L^\infty(\Omega_{R}(x_0))}+C R^{\gamma_0} \big([f_\alpha]_{C^{\gamma_0}(\Omega)}+[g]_{C^{\gamma_0}(\Omega)}\big)
\end{aligned}
$$
for any $x_0\in {\Omega}$ and $R\in (0,R_1]$,
where $C>0$ is a constant depending only on $d$, $\lambda$, $\gamma_0$, $N$, $R_0$, and $[A^{\alpha\beta}]_{C^{\gamma_0}(\Omega)}$.
This completes the proof of the assertion $(b)$ in Theorem \ref{M4}, and that of Theorem \ref{M4}.
\end{proof}

\subsection{Proof of Theorem \ref{M5}}

To prove the theorem, we consider the following two cases:
$$
2\le q<\infty, \quad 1<q<2.
$$
\begin{enumerate}[i.]
\item
$2\le q<\infty$:
We only need to consider the case when $q=2$.
We adapt the arguments in the proof of \cite[Theorem 1.9]{MR3747493},
where the authors proved the weak type-$(1,1)$ estimate for $W^{1,2}$-weak solutions to elliptic equations.
By the hypothesis of the theorem, $\Omega$ is a Lipschitz domain, which implies that the $W^{1,2}_0$-solvability
of the problem
\begin{equation}		\label{171127@eq5}
\left\{
\begin{aligned}
\cL u+\nabla p=D_\alpha f_\alpha \quad &\text{in }\, \Omega\\
\operatorname{div} u=g-(g)_\Omega \quad &\text{in }\, \Omega
\end{aligned}
\right.
\end{equation}
is available (see, for instance, \cite[Lemma 3.2]{MR3693868}).
Define a bounded linear operator $T$ on $L^2(\Omega)^{d\times d}\times L^2(\Omega)$ by
$$
T(f_1,\ldots,f_d,g)=(D_1u,\ldots,D_d u,p),
$$
where $(u,p)\in W^{1,2}_0(\Omega)^d\times \tilde{L}^2(\Omega)$ is the weak solution of \eqref{171127@eq5}.
To get the desired estimate \eqref{180315@eq2}, it suffices to show that $T$ satisfies the hypothesis of the following lemma.

\begin{lemma}		\label{171127@lem1}
Let $\Omega$ be a bounded domain in $\bR^d$ satisfying
\begin{equation}		\label{180315@eq4}
|\Omega_r(x)|\ge A_0 r^d \quad \text{for all }\, x\in \overline{\Omega} \, \text{ and }\, r\in (0, \operatorname{diam}\Omega].
\end{equation}
Let $T$ be a bounded linear operator from $L^2(\Omega)^k$ to $L^2(\Omega)^k$, where $k\in \{1,2,\ldots\}$.
Suppose that for any $x_0\in \Omega$, $0<r<\mu \operatorname{diam}\Omega$, and $g\in \tilde{L}^2(\Omega)^k$ with $\operatorname{supp}g\subset \Omega_r(x_0)$, we have
$$
\int_{\Omega\setminus B_{cr}(x_0)}|Tg|\,dx\le C\int_{\Omega_r(x_0)}|g|\,dx,
$$
where $\mu\in (0,1)$, $c\in (1,\infty)$, and $C\in (0, \infty)$.
Then for any $t>0$ and $f\in L^2(\Omega)^k$, we have
$$
\big|\{x\in \Omega:|Tf(x)|>t\}\big|\lesssim_{d,\Omega, k,\mu,c,C,A_0} \frac{1}{t}\int_{\Omega}|f|\,dx.
$$
\end{lemma}
\begin{proof}
See \cite[Lemma 4.1]{MR3747493}.
\end{proof}

We note that by \eqref{171127@eq3}, $\Omega$ satisfies \eqref{180315@eq4} with $A_0=A_0(d,R_0, \varrho_0, \operatorname{diam}\Omega)$.
We claim that $T$ satisfies the hypothesis of Lemma \ref{171127@lem1} with
$$
\mu=\frac{1}{4}\min\left\{1, \frac{R_1}{\operatorname{diam}\Omega}\right\}, \quad c=4 , \quad C=C(d,\lambda,\Omega, R_0, \varrho_0, \omega_{A^{\alpha\beta}}, C_0)>0.
$$
Here and in this proof, $R_1$, $\kappa$, $\tilde{\varrho}_0$, $\tilde{\omega}_\bullet$, $\varrho^\sharp_0$, and $\omega^\sharp_{\bullet}$ are
 those in the proof of Theorem \ref{M4}.
Fix $x_0\in \Omega$ and $0<r<\mu\operatorname{diam}\Omega$.
Assume that $(u,p)\in W^{1,2}_0(\Omega)^d\times \tilde{L}^2(\Omega)$ is the weak solution of \eqref{171127@eq5}, where $f_\alpha\in \tilde{L}^2(\Omega)^d$ and $g\in \tilde{L}^2(\Omega)$ are supported in $\Omega_r(x_0)$.
Let $R\in [4 r, \operatorname{diam}\Omega)$ so that $\Omega\setminus B_R(x_0)\neq \emptyset$, and let $\cL^*$ be the adjoint operator of $\cL$, i.e.,
$$
\cL^*v=D_\alpha(A^{\alpha\beta}_*D_\beta v), \quad A^{\alpha\beta}_*=(A^{\beta \alpha})^{\top}.
$$
Then by \cite[Lemma 3.2]{MR3693868}, for given
$$
\phi_\alpha\in C^\infty_0(\Omega_{2R}(x_0)\setminus B_R(x_0))^d, \quad \psi\in C^\infty_0(\Omega_{2R}(x_0)\setminus B_R(x_0)),
$$
there exists a unique $(v,\pi)\in W^{1,2}_0(\Omega)^d\times \tilde{L}^2(\Omega)$ satisfying
\begin{equation}		\label{171127@eq6a}
\left\{
\begin{aligned}
\cL^*v+\nabla \pi=D_\alpha \phi_\alpha \quad \text{in }\, \Omega,\\
\operatorname{div} v=\psi-(\psi)_{\Omega} \quad \text{in }\, \Omega,
\end{aligned}
\right.
\end{equation}
and
\begin{equation}		\label{171127@eq6}
\||Dv|+|\pi|\|_{L^2(\Omega)}\lesssim_{d,\lambda,\Omega} \||\phi_\alpha|+|\psi|\|_{L^2(\Omega_{2R}(x_0)\setminus B_R(x_0))}.
\end{equation}
By applying $u$ and $v$ as test functions to \eqref{171127@eq6a} and \eqref{171127@eq5}, respectively, we have
\begin{equation}		\label{171127@eq6b}
\begin{aligned}
&\int_\Omega  (D_\alpha u\cdot \phi_\alpha+p \psi )\,dx\\
&=\int_{\Omega_r(x_0)} \big(D_\alpha v- (D_\alpha v)_{\Omega_r(x_0)}\big)\cdot  f_\alpha+\big(\pi-(\pi)_{\Omega_r(x_0)}\big) g\,dx.
\end{aligned}
\end{equation}
Observe that
$$
4r \le \min\{R_1, R\}<\operatorname{diam}\Omega.
$$
Since $\phi_\alpha=\psi=0$ in $\Omega_R(x_0)$, by \eqref{171128@C1}, \eqref{180315@eq1}, and H\"older's inequality, we obtain that for any $x,y\in \Omega_r(x_0)$,
\begin{equation}		\label{180320@A1}
\begin{aligned}
&|Dv(x)-Dv(y)|+|\pi (x)- \pi(y)|\\
&\le C  R^{-d/2}\||Dv|+|\pi|\|_{L^2(\Omega_{R}(x_0))}\bigg(\left(\frac{r}{R}\right)^{\gamma} +\int_0^{2r} \frac{\varrho^\sharp_0(t)+\omega^\sharp_{A^{\alpha\beta}}(t)}{t}\,dt\bigg),
\end{aligned}
\end{equation}
where $\gamma=1/2$ and  $C=C(d,\lambda, \Omega,R_0, \varrho_0, \omega_{A^{\alpha\beta}})$.
Combining \eqref{171127@eq6} -- \eqref{180320@A1}, and then using the duality, we see that
\begin{equation}		\label{180315@eq6}
\int_{\Omega_{2R}(x_0)\setminus B_R(x_0)} (|Du|+|p|)\,dx \lesssim  M\bigg(\left(\frac{r}{R}\right)^{\gamma} +\int_0^{2r} \frac{\varrho^\sharp_0(t)+\omega^\sharp_{A^{\alpha\beta}}(t)}{t}\,dt\bigg),
\end{equation}
where we set
$$
M=\int_{\Omega_r(x_0)}( |f_\alpha|+|g|)\,dx.
$$
Notice from \eqref{171127@B1} and \cite[Eq. (3.5)]{MR3620893} that
$$
\tilde{\varrho}_0(\rho)+\tilde{\omega}_{A^{\alpha\beta}}(\rho)\le C (\ln \rho)^{-2}, \quad \forall \rho\in (0,1/2),
$$
where $C=C(\gamma,\kappa,C_0)=C(d,\lambda,R_0, \varrho_0, C_0)$.
Then it is routine to verify that
$$
{\varrho}^\sharp_0(\rho)+{\omega}^\sharp_{A^{\alpha\beta}}(\rho)\le C(\ln \rho)^{-2}, \quad \forall \rho\in (0,R_1],
$$
and thus, we have
$$
\int_0^{2r} \frac{\varrho^\sharp_0(t)+\omega^\sharp_{A^{\alpha\beta}}(t)}{t}\,dt
\lesssim \left(\ln \frac{1}{r}\right)^{-1}.
$$
This inequality together with \eqref{180315@eq6} yields
$$
\int_{\Omega_{2R}(x_0)\setminus B_R(x_0)} (|Du|+|p|)\,dx \lesssim \bigg(\left(\frac{r}{R}\right)^\gamma +\left(\ln \frac{1}{r}\right)^{-1}\bigg) M.
$$
Let $N$ be the smallest positive integer such that $\Omega\subset B_{2^{N+1}r}(x_0)$.
By taking $R=2^{i+1} r$, $i\in \{1,2,\ldots,N-1\}$, and using $N-1\lesssim \ln (1/r)$, we have
$$
\int_{\Omega\setminus B_{4 r}(x_0)} (|Du|+|p|)\,dx\le C \sum_{k=1}^{N-1} \big(2^{-k\gamma}+(\ln(1/r))^{-1}\big)M\le C M,
$$
where $C=C(d,\lambda,\Omega, R_0, \varrho_0, \omega_{A^{\alpha\beta}}, C_0)$.
Therefore, the map $T$ satisfies the hypothesis of Lemma \ref{171127@lem1}.
\item
$1<q<2$: In this case, we use an approximation argument together with the result in the first case, and
 the $W^{1,q}$-estimate for the Stokes system in \cite{MR3693868} (see also \cite{MR3758532}).
By \cite[Theorem 5.1 and Corollary 5.3]{MR3693868}, the $W^{1,q}$-estimate and solvability are available when the domain $\Omega$ has Lipschitz boundary with a small Lipschitz constant and the coefficients $A^{\alpha\beta}$ have vanishing mean oscillations (VMO):
\begin{equation}		\label{180316@eq1}
\lim_{\delta\to 0} \sup_{x\in \overline{\Omega}} \sup_{r\in (0, \delta]}\dashint_{B_r(x)}|A^{\alpha\beta}-(A^{\alpha\beta})_{B_r(x)}|\,dy=0.
\end{equation}
The coefficients $A^{\alpha\beta}$ considered in this paper are VMO in the sense that  (see Remark \ref{180213@rmk1})
\begin{equation}		\label{180213@A1}
\lim_{\delta\to 0} \sup_{x\in \overline{\Omega}} \sup_{r\in (0, \delta]}\dashint_{\Omega_r(x)}|A^{\alpha\beta}-(A^{\alpha\beta})_{\Omega_r(x)}|\,dy=0,
\end{equation}
which is slightly weaker than \eqref{180316@eq1}.
However, it is easy to check that the proofs of \cite[Theorem 5.1 and Corollary 5.3]{MR3693868} still work under the condition \eqref{180213@A1}.

Now, we are ready to prove \eqref{180315@eq2} when $q\in (1,2)$.
Assume that $(u, p)\in W^{1,q}_0(\Omega)^d\times \tilde{L}^q(\Omega)$ is the weak solution of \eqref{171006@eq2}, where $f_\alpha\in L^q(\Omega)^d$ and $g\in L^q(\Omega)$.
Let $\{f_{\alpha, k}\}\subset L^2(\Omega)^d$ and $\{g_k\}\subset L^2(\Omega)$ be sequences such that
\begin{equation}		\label{180316@eq5}
f_{\alpha, k}\to f_\alpha, \quad g_k \to g \quad \text{in }\, L^q(\Omega) \, \text{ as }\, k\to \infty.
\end{equation}
By the $W^{1,2}_0$-solvability of the problem \eqref{171006@eq2}, for $k\in \{1,2,\ldots,\}$, there exists a unique weak solution $(u_k, p_k)\in W^{1,2}_0(\Omega)^d\times \tilde{L}^2(\Omega)$ of \eqref{171006@eq2} with $f_{\alpha, k}$ and $g_k$ in place of $f_\alpha$ and $g$.
Then by the result in the first case, we see that
$$
\big|\{x\in \Omega:|Du_k(x)|+|p_k(x)|>t\}\big|\le \frac{C'}{t}\int_\Omega (|f_{\alpha,k}|+|g_k| )\,dx, \quad \forall t>0,
$$
where $C'=C'(d,\lambda, \Omega, R_0, \varrho_0, \omega_{A^{\alpha\beta}}, C_0)$.
Moreover, since $(u-u_k, p-p_k)\in W^{1,q}_0(\Omega)^d\times \tilde{L}^q(\Omega)$ satisfies
$$
\left\{
\begin{aligned}
\cL (u-u_k)+\nabla (p-p_k)=D_\alpha (f_\alpha-f_{\alpha,k}) \quad &\text{in }\, \Omega,\\
\operatorname{div} (u-u_k)=g-g_k-(g)_\Omega +(g_k)_{\Omega} \quad &\text{in }\, \Omega,
\end{aligned}
\right.
$$
by the $W^{1,q}$-estimate and \eqref{180316@eq5}, we have
$$
\begin{aligned}
&\|Du-Du_k\|_{L^q(\Omega)}+\|p-p_k\|_{L^q(\Omega)}\\
&\lesssim \|f_\alpha-f_{\alpha,k}\|_{L^q(\Omega)}+\|g-g_k\|_{L^q(\Omega)} \to 0 \quad \text{as }\, k\to \infty.
\end{aligned}
$$
Observe that
$$
\begin{aligned}
&\big|\{x\in \Omega:|Du(x)|+|p(x)|>t\}\big|\\
&\le \big|\{x\in \Omega:|Du_k(x)|+|p_k(x)|>t/2\}\big|\\
&\quad +\big|\{x\in \Omega:|Du(x)-Du_k(x)|+|p(x)-p_k(x)|>t/2\}\big|\\
&\lesssim_{C'} \frac{1}{t}\int_\Omega (|f_{\alpha,k}|+|g_k|)\,dx+\frac{1}{t^q}\int_\Omega (|Du-Du_k|+|p-p_k|)^q\,dx.
\end{aligned}
$$
Since the right-hand side of the above inequality converges to
$$
\frac{1}{t}\int_\Omega (|f_{\alpha}|+|g|)\,dx,
$$
we get the desired estimate \eqref{180315@eq2}.
\end{enumerate}
The theorem is proved.
\qed

\section{Appendix}		\label{Sec3}

In Appendix, we provide the proofs of some lemmas used in the previous section.

\begin{lemma}		\label{171024@lem1}
Let $\omega:(0,a]\to [0,\infty)$ be a Dini function satisfying \eqref{171006@eq1} and \eqref{180315@A1}.
Set
$$
\tilde{\omega}(r):=\sum_{i=1}^\infty \kappa^{\gamma i}\big(\omega(\kappa^{-i}r)[\kappa^{-i}r<a]+\omega(a)[\kappa^{-i}r\ge a]\big),
$$
where $\gamma\in (0,1)$ and $\kappa\in (0,1/2]$.
Then $\tilde{\omega}:(0,a]\to [0,\infty)$ is also a Dini function satisfying
\begin{equation}		\label{171024@eq2}
\tilde{\omega}(t)\lesssim_{c_1} \tilde{\omega}(s) \lesssim_{c_2} \tilde{\omega}(t) \quad \text{whenever }\, \frac{t}{2}\le s\le t\le a
\end{equation}
and that
\begin{equation}		\label{171101@eq1}
\int_0^a \frac{\tilde{\omega}(t)}{t}\,dt<\infty.
\end{equation}
\end{lemma}

\begin{proof}
Set
$$
\hat{\omega}(r)=
\left\{
\begin{aligned}
\omega(r) &\quad \text{if }\, r<a,\\
\omega(a) &\quad \text{if }\, r\ge a,
\end{aligned}
\right.
$$
and observe that
$$
\tilde{\omega}(r)=\sum_{i=1}^\infty \kappa^{\gamma i}\hat{\omega}(\kappa^{-i}r).
$$
Let $\frac{t}{2}\le s\le t\le a$.
To prove \eqref{171024@eq2}, it suffices to show that for any $i\in \{1,2,\ldots\}$, we have
\begin{equation}		\label{180103@eq1a}
\hat{\omega}(\kappa^{-i}t)\lesssim_{c_1} \hat{\omega}(\kappa^{-i}s)\lesssim_{c_2}\hat{\omega}(\kappa^{-i}t).
\end{equation}
For $i$ satisfying $\kappa^{-i}t< a$, by \eqref{171006@eq1} and the fact that
$$
\frac{\kappa^{-i}t}{2}\le \kappa^{-i}s\le \kappa^{-i} t,
$$
we have
$$
\hat{\omega}(\kappa^{-i}t)=\omega(\kappa^{-i}t)\lesssim_{c_1} \omega(\kappa^{-i}s)=\hat{\omega}(\kappa^{-i}s)\lesssim_{c_2} \omega(\kappa^{-i}t)=\hat{\omega}(\kappa^{-i}t),
$$
which gives \eqref{180103@eq1a}.
On the other hand, for $i$ satisfying $\kappa^{-i} t\ge a$, we consider the two cases:
$$
\kappa^{-i}s< a, \quad \kappa^{-i}s\ge a.
$$
If $\kappa^{-i}s< a$, then by \eqref{171006@eq1} and the fact that
$$
\frac{a}{2}\le \kappa^{-i}s< a,
$$
we have
$$
\hat{\omega}(\kappa^{-i}t)=\omega(a)\lesssim_{c_1}\omega(\kappa^{-i}s)=\hat{\omega}(\kappa^{-i}s) \lesssim_{c_2} \omega(a)=\hat{\omega}(\kappa^{-i}t),
$$
which implies \eqref{180103@eq1a}.
If $\kappa^{-i}s\ge a$, then by the definition of $\hat{\omega}$, we obtain that
$$
\hat{\omega}(\kappa^{-i}t)=\hat{\omega}(\kappa^{-i}s).
$$
Thus we prove that \eqref{180103@eq1a} holds.
For the proof of \eqref{171101@eq1},
we refer to \cite[Lemma 1]{MR2927619}.
The lemma is proved.
\end{proof}

\begin{lemma}		\label{180213@lem1}
Let $\omega:(0,a]\to [0,\infty)$ be a Dini function satisfying \eqref{171006@eq1} and \eqref{180315@A1}.
Then for any $\varepsilon>0$, there exists $\delta\in (0,1)$, depending only on $c_1$ and $\varepsilon$, such that
$$
\sup_{r\in (0, \delta]} \omega(r)<\varepsilon.
$$
\end{lemma}

\begin{proof}
Observe that
$$
\omega(r)\le C_0 \inf_{s\in[r/2,r]}\omega (s)\le C_0 \int_{r/2}^r \frac{\omega (s)}{s}\,ds
$$
for any $r\in (0,a]$, where $C_0=C_0(c_1)$.
Therefore, for given $\varepsilon>0$, if we take $\delta=\delta(c_1,\varepsilon)>0$ such that
$$
\int_0^\delta \frac{\omega(s)}{s}\,ds<\frac{\varepsilon}{C_0},
$$
then $\omega(r)<\varepsilon$ for all $r\in (0,\delta]$.
\end{proof}

\begin{remark}		\label{180213@rmk1}
From Remark \ref{171020@rmk1} and Lemma \ref{180213@lem1}, it follows that if $f$ is of Dini mean oscillation in $\Omega$ satisfying Definition \ref{D2} $(ii)$, then $f$ has vanishing mean oscillation in the sense that
$$
\lim_{\delta\to 0} \sup_{x\in \overline{\Omega}} \sup_{r\in (0, \delta]} \dashint_{\Omega_r(x)} |f-(f)_{\Omega_r(x)}|\,dy=0.
$$
\end{remark}

\bibliographystyle{plain}

\end{document}